\documentclass[12pt]{amsart}
\usepackage[margin=1.2in]{geometry}
\usepackage[T1]{fontenc}
\usepackage[utf8]{inputenc}
\usepackage{graphicx,latexsym}
\usepackage{tikz-cd}
\usepackage{tikz}
\usepackage[all]{xy}
\usepackage{hyperref}
\usepackage{amsfonts, amssymb, amsmath, amsthm, bm, mathtools}
\usepackage[inline]{enumitem}
\usepackage{cleveref}
\usepackage{setspace}
\usepackage[mathscr]{euscript}

\newtheorem{theorem}{Theorem}[section]
\newtheorem{proposition}[theorem]{Proposition}
\newtheorem{lemma}[theorem]{Lemma}
\newtheorem{corollary}[theorem]{Corollary}

\hypersetup{
    colorlinks,
    linkcolor={blue!80!black},
    citecolor={blue!80!black},
    urlcolor={blue!80!black}
}

\usepackage[addedmarkup=none,deletedmarkup=sout]{changes}
\setdeletedmarkup{{\color{red}\sout{#1}}}

\theoremstyle{definition}
\newtheorem{definition}[theorem]{Definition}
\theoremstyle{remark}
\newtheorem{remark}[theorem]{Remark}

\numberwithin{equation}{section}

\DeclareMathOperator{\Vol}{Vol}
\DeclareMathOperator{\supp}{supp}
\DeclareMathOperator{\TO}{TO}
\DeclareMathOperator{\RO}{RO}
\DeclareMathOperator{\NO}{NO}
\DeclareMathOperator{\TOloc}{\TO_{\mathrm{loc}}}
\DeclareMathOperator{\TOzloc}{\TO^0_{\mathrm{loc}}}
\DeclareMathOperator{\ROloc}{\RO_{\mathrm{loc}}}

\DeclareMathOperator{\tTOloc}{\widetilde{\TO}_{\mathrm{loc}}}
\DeclareMathOperator{\tTOzloc}{\widetilde{\TO}^0_{\mathrm{loc}}}
\DeclareMathOperator{\tROloc}{\widetilde{\RO}_{\mathrm{loc}}}

\newcommand{\ud}{\mathrm{d}}
\newcommand{\N}{\mathbb N}

\newcommand{\C}{\mathbb C}
\newcommand{\R}{\mathbb R}
\newcommand{\Z}{\mathbb Z}

\newcommand{\DD}{\mathcal{D}}
\newcommand{\cM}{\mathcal{M}}
\newcommand{\cS}{\mathcal{S}}
\newcommand{\cN}{\mathcal{N}}

\newcommand{\dx}{{\rm d}x }
\newcommand{\dxi}{{\rm d}\xi }
\newcommand{\dt}{{\rm d}t }
\newcommand{\deta}{{\rm d}\eta }

\DeclareMathOperator{\id}{id}
\DeclareMathOperator{\csn}{csn}
\DeclareMathOperator{\Hom}{Hom}
\newcommand{\cD}{\mathcal{D}}

\let\Linc\cL
\newcommand{\cE}{\mathcal{E}}
\newcommand{\cG}{\mathcal{G}}
\newcommand{\cI}{\mathcal{I}}
\newcommand{\cA}{\mathcal{A}}
\newcommand{\coleq}{\coloneqq}
\newcommand{\e}{\varepsilon}

\newcommand{\cEloc}{\cE_{\textrm{loc}}}
\newcommand{\cGloc}{\cG_{\textrm{loc}}}
\newcommand{\cGsloc}{\cG^*_{\textrm{loc}}}

\usetikzlibrary{arrows,chains,matrix,positioning,scopes}
\makeatletter
\tikzset{join/.code=\tikzset{after node path={%
\ifx\tikzchainprevious\pgfutil@empty\else(\tikzchainprevious)%
edge[every join]#1(\tikzchaincurrent)\fi}}}
\makeatother
\tikzset{>=stealth',every on chain/.append style={join},
         every join/.style={->}}
\tikzstyle{labeled}=[execute at begin node=$\scriptstyle,
   execute at end node=$]

\begin{document}
\author[A.~Debrouwere]{Andreas Debrouwere}
\address{Department of Mathematics, Ghent University, Krijgslaan 281, 9000 Gent, Belgium}
\email{Andreas.Debrouwere@UGent.be}
\author[E.A.~Nigsch]{Eduard A.~Nigsch}
\address{Wolfgang Pauli Institute, Oskar-Morgenstern-Platz 1, 1090 Vienna, Austria}\email{eduard.nigsch@univie.ac.at}
\title{Sheaves of nonlinear generalized function spaces}
\date{July 2016}
\thanks{A. Debrouwere gratefully acknowledges support by Ghent University, through a BOF Ph.D.-grant} 
\thanks{E.~A.~Nigsch~was supported by grant P26859 of the Austrian Science Fund (FWF)}
\begin{abstract}
We provide a framework for the construction of diffeomorphism invariant sheaves of nonlinear generalized functions spaces. As an application, global algebras of generalized functions for distributions on manifolds and diffeomorphism invariant algebras of generalized functions for ultradistributions are constructed.
\end{abstract}
\maketitle

\section{Introduction}
The theory of generalized functions developed by L.~Schwartz \cite{TD} suffers from the fact that in general one cannot define nonlinear operations (like multiplication) on distributions, so the use of this theory for nonlinear problems is limited. In the 1980s, differential algebras of nonlinear generalized functions were developed by J.~F.~Colombeau \cite{ColNew,ColElem} in order to study nonlinear PDEs with singular data or coefficients. These Colombeau algebras have found numerous applications for instance in connection with PDEs involving singular data and/or coefficients, singular differential geometry and general relativity. In particular, a diffeomorphism invariant formulation of the theory was developed in \cite{found,global} and recently extended to the vector-valued setting in \cite{bigone}.

Spaces of nonlinear generalized ultradistributions were studied in \cite{zbMATH06071218,zbMATH02134784,1126.46030,zbMATH04205141,zbMATH01618580}. The most recent variant, developed in \cite{D-V-V}, is optimal in the sense that the embedding of ultradistributions of class $M_p$ there preserves the product of all ultradifferentiable functions of class $M_p$; this is an improvement over the previous variants where only the product of ultradifferentiable functions of a strictly more regular class had been preserved. The setting of \cite{D-V-V} is that of \emph{special} Colombeau algebras, which allows for a simpler development of the theory but makes it impossible to obtain diffeomorphism invariance (cf.~\cite[Chapter 2]{GKOS}). \emph{Full} Colombeau algebras, on the other hand, are technically more involved but allow for an embedding of distributions that is diffeomorphism invariant and in addition commutes with arbitrary derivatives (cf.~\cite{found,global, papernew}). Hence, for applications in a geometric context it is essential that a formulation of the theory in the full setting is obtained.

As a continuation of the development of \cite{papernew,gf14proc} and \cite{D-V-V} we work out the abstract formulation of the construction of sheaves of nonlinear generalized function spaces. In particular, it turns out that very little structure is needed on the underlying spaces of generalized functions as the respective arguments mainly concern the sheaf structure.

Our construction applies at the same time to distributions and to ultradistributions, both of Beurling and Roumieu type, which leads to the following results. 

\begingroup
\def\thetheorem{\ref{thmone}}
\begin{theorem}Let $M$ be a paracompact Hausdorff manifold. There is an associative commutative algebra $\cGloc(M)$ with unit containing $\cD'(M)$ injectively as a linear subspace and $C^\infty(M)$ as a subalgebra. $\cGloc(M)$ is a differential algebra, where the derivations $\widehat L_X$ extend the usual Lie derivatives from $\cD'(M)$ to $\cGloc(M)$, and $\cGloc$ is a fine sheaf of algebras over $M$.
\end{theorem}
\addtocounter{theorem}{-1}
\endgroup

As customary, we write $\ast$ instead of $(M_p)$ or $\{M_p\}$ to treat the Beurling and Roumieu case simultaneously. For the following theorem, let $M_p$ be a weight sequence satisfying $(M.1)$, $(M.2)$, and $(M.3)'$.
\begingroup
\def\thetheorem{\ref{thmtwo}}
\begin{theorem}For each open set $\Omega \subseteq \R^n$ there is an associative commutative algebra with unit $\cGsloc(\Omega)$ containing $\cD^{*\prime}(\Omega)$ injectively 
as a linear subspace and $\cE^*(\Omega)$ as a subalgebra. $\cGsloc(\Omega)$ is a differential algebra, where the partial derivatives $\widehat \partial_i$, $i = 1,\ldots,n$, extend the usual partial derivatives from $\cD^\ast(\Omega)$ to $\cGsloc(\Omega)$, and $\cGsloc$ is a fine sheaf of algebras over $\Omega$. Moreover, the construction is invariant under real-analytic coordinate changes, i.e., if $\mu \colon \Omega' \to \Omega$ is a real-analytic diffeomorphism then there is a map $\widehat \mu 
\colon \cGsloc(\Omega') \to \cGsloc(\Omega)$ compatible with the canonical embeddings $\iota$ and $\sigma$. 
\end{theorem}
\addtocounter{theorem}{-1}
\endgroup

The structure of this article is as follows.
\begin{itemize}
 \item We collect some preliminary notions in \Cref{sec_preliminaries}.
 \item The basic spaces containing the representatives of nonlinear generalized functions are introduced in \Cref{global-theory}.
 \item The quotient construction, which ensures that the product of smooth or of ultradifferentiable functions is preserved, is detailed in \Cref{sec_quotient}.
 \item Sheaf properties of the quotient space are established in \Cref{local-theory}.
 \item The construction of diffeomorphism invariant differential algebras of distributions and ultradistributions is given in \Cref{sec_appldist} and \Cref{sec_applications}, respectively.
\end{itemize}

\section{Preliminaries}\label{sec_preliminaries}

Our general references are \cite{TD} for distribution theory, \cite{Komatsu,Komatsu2,Komatsu3} for ultradistributions and \cite{ColNew, ColElem, MOBook,GKOS} for Colombeau algebras.

We set $I=(0,1]$, $\R_+ = [0, \infty)$ and $\N = \{0,1,2,\dotsc\}$. Given a set $M$, $\id_M$ (or simply $\id$ if the set is clear from the context) denotes the identity mapping on $M$. For an element $\lambda \in (\R_+)^I$ we write $\lambda(\e) = \lambda_\e$. Furthermore, Landau's $O$-notations are always meant for $\e \to 0^+$. Given two locally convex spaces $E$ and $F$, $\Linc_b(E,F)$ denotes the space of continuous linear mappings from $E$ into $F$ endowed with the topology of bounded convergence. $\Linc_\sigma(E,F)$ is this space endowed with the weak topology instead. We denote by $\csn(E)$ the set of continuous seminorms on $E$. An algebra always means an associative commutative algebra over $\C$, and a locally convex algebra is an algebra endowed with a locally convex topology such that its multiplication is jointly continuous. $C^\infty(E,F)$ is the space of smooth functions $E \to F$ in the sense of convenient calculus \cite{KM}, with $C^\infty(E) \coleq C^\infty(E, \C)$; in this context, $\ud^k f$ denotes the $k$th differential of a mapping $f \in C^\infty(E, F)$.

Colombeau algebras are usually defined by means of a quotient construction employing certain asymptotic scales. Most frequently a polynomial scale is used for this purpose, but we will employ more general scales based on \cite{G-B} instead, which increases the flexibility regarding applications.

\begin{definition} A set $\cA \subseteq (\R_+)^I$ is said to be an \emph{asymptotic growth scale} if
\begin{enumerate}[label=(\roman*)]
\item $\forall \lambda, \mu \in \mathcal {A}$ $\exists \nu \in \cA$: $\lambda_\e + \mu_\e = O(\nu_\e)$,
\item $\forall \lambda, \mu \in \mathcal {A}$ $\exists \nu \in \cA$: $\lambda_\e \mu_\e = O(\nu_\e)$,
\item  $\exists \lambda \in \cA$: $\displaystyle \liminf_{\e \to 0^+} \lambda _\e > 0$.
\end{enumerate}
A set $\cI \subseteq (\R_+)^I$ is said to be an \emph{asymptotic decay scale} if
\begin{enumerate}[label=(\roman*),resume]
\item $\forall \lambda \in \mathcal {I}$ $\exists \mu, \nu \in \cI$: $\mu_\e + \nu_\e = O(\lambda_\e)$,
\item $\forall \lambda \in \mathcal {I}$ $\exists \mu, \nu \in \cI$: $\mu_\e \nu_\e = O(\lambda_\e)$,
\item $\exists \lambda \in \cI$: $\displaystyle \lim_{\e \to 0^+} \lambda _\e = 0$.
\end{enumerate}
We call a pair $(\cA, \cI)$ an \emph{admissible pair of scales} if $\cA$ is an asymptotic growth scale, $\cI$ is an asymptotic decay scale, and the following two properties are satisfied:
\begin{enumerate}[label=(\roman*),resume]
\item $\forall \lambda \in \cI$ $\forall \mu \in \cA$ $\exists \nu \in \cI$: $\mu_\e \nu_\e = O(\lambda_\e)$,
\item $\exists \lambda \in \cA$ $\exists \mu \in \cI$: $\mu_\e = O(\lambda_{\e})$.
\end{enumerate}
\end{definition}

The prototypical scale to keep in mind is given by the polynomial scale
\begin{equation}\label{polyscale}
\cA = \cI = \{\, \e \mapsto \e^k\ |\ k \in \Z\,\},
\end{equation}
which is easily verified to give an admissible pair. For a detailed study of asymptotic scales we refer to \cite{1126.46030,G-B}. 

\section{The basic space}\label{global-theory}
A main principle behind Colombeau algebras is to represent singular functions by regular ones and thus define classical operations like multiplication on the former through the latter. Usually the roles of singular and regular functions are played by $\DD'$ and $C^\infty$, respectively, but for our considerations we will replace these spaces by a more general pair of locally convex spaces $E$ and $F$. Such a pair $(E,F)$ is called a \emph{test pair} if $F \subseteq E$ and the topology on $F$ is finer than the one induced by $E$. Throughout this section we fix a test pair $(E,F)$.
\begin{definition}\label{basic-space} We define the \emph{basic space} as
\[ \cE(E,F) \coleq C^\infty(\Linc_b(E,F), F) \]
and the canonical linear embeddings of $E$ and $F$ into $\cE(E,F)$ via
\begin{gather*}
\iota \colon  E \to \cE(E,F), \quad \iota(u)(\Phi) \coleq \Phi(u), \\
\sigma \colon  F \to \cE(E,F), \quad \sigma(\varphi)(\Phi) \coleq \varphi.
\end{gather*}
\end{definition}

There are three common ways of transferring classical operations $T$ on $E$ and $F$ to elements $R$ of the basic space $\cE(E,F)$. These are, in brief, given as follows:
\begin{align*}
 (\widetilde T R)(\Phi) &\coleq T ( R (\Phi)), \\
 (T_* R)(\Phi) &\coleq T ( R ( T^{-1} \circ \Phi \circ T)),\\
 (\widehat T R)(\Phi) &\coleq -\ud R(\Phi) ( T \circ \Phi - \Phi \circ T) + T ( R ( \Phi )).
\end{align*}

We will now specify in which situation they are well-defined on the basic space, and when each variant is employed.

The first one amounts to applying an operation on $F$ after inserting the parameter $\Phi \in \Linc(E,F)$. This defines the vector space structure of $\cE(E,F)$ and its algebra structure if $F$ is a locally convex algebra. Moreover, this is used for extending directional derivatives and especially the covariant derivative in geometry (see \cite{bigone}). For multilinear mappings it is  formulated as follows:

\begin{lemma}\label{extension-bs}
Let $T \colon F \times \cdots \times F \to F$ be a jointly continuous multilinear mapping. Then, the mapping $\widetilde T \colon \cE(E,F) \times \cdots \times \cE(E,F)\to \cE(E,F)$ given by
\begin{equation}
\widetilde T(R_1, \ldots, R_n)(\Phi) \coleq T(R_1(\Phi), \ldots , R_n(\Phi))
\label{definition1}
\end{equation}
commutes with the embedding $\sigma$ in the sense that
\[
\widetilde T(\sigma(\varphi_1), \ldots, \sigma(\varphi_n)) = \sigma(T(\varphi_1, \ldots, \varphi_n)).
\]
\end{lemma}

\begin{corollary}
Suppose that $F$ is a locally convex algebra. Then, $\cE(E,F)$ is an algebra with multiplication given by
\begin{equation}
(R_1 \cdot R_2)(\Phi) \coleq R_1(\Phi) \cdot R_2(\Phi)
\label{definition-multiplication}
\end{equation}
and $\sigma$ is an algebra homomorphism.
\end{corollary}
The second variant of extending operations to the basic space applies to isomorphisms on $E$ which restrict to isomorphisms on $F$. This will be used for isomorphisms on distribution spaces coming from diffeomorphisms of the respective domains.

\begin{lemma}\label{extension-diff}
Let $(E_1,F_1)$ and $(E_2,F_2)$ be two test pairs. Suppose that $f \colon E_1 \to E_2$ is a linear topological isomorphism such that also the restriction $f|_{F_1}$ is a linear topological isomorphism $F_1 \to F_2$. Then, the mapping $f_* \colon \cE(E_1,F_1) \to \cE(E_2,F_2)$ given by
\begin{equation}
(f_*R)(\Phi) \coleq f(R(f^{-1} \circ \Phi \circ f))
\label{definition2}
\end{equation}
is a vector space isomorphism that makes the following diagrams commutative:
 \[
 \xymatrix{
  E_1 \ar[r]^f \ar[d]_\iota & E_2 \ar[d]^\iota \\
  \cE(E_1, F_1) \ar[r]^f & \cE(E_2, F_2)
  }
  \hspace{1cm}
  \xymatrix{
  F_1 \ar[r]^f \ar[d]_\sigma & F_2 \ar[d]^\sigma \\
  \cE(E_1, F_1) \ar[r]^f & \cE(E_2, F_2)
  }
\]

\end{lemma}

Finally, the third variant of extending operations to the basic space applies to the extension of derivatives to $\cE(E,F)$:
\begin{lemma}\label{extension-der} Let $T \in \Linc(E,E)$ with $T|_{F} \in \Linc(F, F)$. Then, the mapping 
\begin{gather*}
T^{\RO} \colon \Linc(E,F) \to \Linc(E,F), \\
\Phi \mapsto  T \circ \Phi - \Phi \circ T
\end{gather*}
is linear and continuous, and the mapping $\widehat T \colon \cE(E,F) \to \cE(E,F)$ given by
\begin{equation}\label{definition3}
(\widehat TR)(\Phi) \coleq T(R(\Phi)) - \ud R (\Phi) (T^{\RO}\Phi)
\end{equation}
is a well defined linear mapping that makes the following diagrams commutative:
 \[
 \xymatrix{
  E \ar[r]^T \ar[d]_\iota & E \ar[d]^\iota \\
  \cE(E, F) \ar[r]^{\widehat{T}}& \cE(E, F)
  }
  \hspace{1cm}
  \xymatrix{
  F \ar[r]^T \ar[d]_\sigma & F \ar[d]^\sigma \\
  \cE(E, F) \ar[r]^{\widehat{T}} & \cE(E, F)
  }
 \]
\end{lemma}

\section{The quotient construction}\label{sec_quotient}

Colombeau algebras are defined as the quotient of moderate by negligible functions, which permits the product of regular functions to be preserved. While originally these properties were determined by inserting translated and scaled test functions into the representatives of generalized functions, the functional analytic formulation of the theory makes it possible to give a very elegant formulation of this testing procedure in more general terms. Our next goal is to give a proper definition of moderateness and negligibility of elements of the basic space in our setting. We start by introducing test objects for a test pair $(E, F)$.

\begin{definition} Let $\cS = (\cA, \cI)$ be an admissible pair of scales. We define $\TO(E,F, \cS)$ as the set consisting of all $(\Phi_\e)_\e \in \Linc(E,F)^I$ that satisfy
\begin{enumerate}[label=$(\TO)_\arabic*$]
\item $\forall p \in \csn(\Linc_\sigma(E,F)) \, \exists \lambda \in \cA : p(\Phi_\e) = O(\lambda_\e)$, 
\item $\forall p \in \csn(\Linc_\sigma(F,F)) \, \forall \lambda \in \cI : p(\Phi_{\e}|_F - \id_F) = O(\lambda_\e)$,
\item $\Phi_\e \to  \id_E$ in $\Linc_\sigma(E,E)$.
\end{enumerate}
Elements of $\TO(E,F, \cS)$ are called \emph{test objects (with respect to $\cS$)}. If $\cS$ is clear from the context, we shall simply write $\TO(E,F, \cS) = \TO(E,F)$.

Similarly, we define $\TO^0(E,F) = \TO^0(E,F, \cS)$ as the set consisting of all $(\Psi_\e)_\e \in \Linc(E,F)^I$ that satisfy
\begin{enumerate}[label=$(\TO)^0_\arabic*$]
\item $\forall p \in \csn(\Linc_\sigma(E,F)) \, \exists \lambda \in \cA : p(\Psi_\e) = O(\lambda_\e)$, 
\item $\forall p \in \csn(\Linc_\sigma(F,F)) \, \forall \lambda \in \cI : p(\Psi_{\e}|_F) = O(\lambda_\e)$,
\item $\Psi_\e \to  0$ in $\Linc_\sigma(E,E)$.
\end{enumerate}
Elements of $\TO^0(E,F, \cS)$ are called \emph{$0$-test objects (with respect to $\cS$)}. Again, we write $\TO^0(E,F,\cS) = \TO^0(E,F)$ if $\cS$ is clear from the context.
\end{definition}
We shall need the following result later on.
\begin{lemma}\label{mappings-TO}
\begin{enumerate}[label=(\roman*)]
\item Let $T_i \in \Linc(E, E)$, $i = 0, \ldots, N \in \N$, be given such that $T_{i}|_F  \in \Linc(F,F)$ and $\sum_{i = 0}^N T_i = \id$. Then, $\left(\sum_{i = 0}^N T_i \circ \Phi_{i,\e}\right)_\e \in \TO(E,F)$ for all $(\Phi_{i,\e})_\e \in \TO(E,F)$, $i= 0, \ldots, N$.
\item Let $T \in \Linc(E, E)$ be such that $T|_F  \in \Linc(F,F)$. Then, $(T \circ \Phi_\e)_\e \in \TO^0(E,F)$ for all $(\Phi_\e)_\e \in \TO^0(E,F)$.
\item Let $T \in \Linc(E,E)$ with $T|_F \in \Linc(F,F)$. Then, $(T \circ \Phi_\e - \Phi_\e \circ T)_\e \in \TO^0(E,F)$ for all $(\Phi_\e)_\e \in \TO(E,F) \cup \TO^0(E,F)$.
\end{enumerate}
\end{lemma}
Having test objects at our disposal, we are now able to define moderateness and negligibility.
\begin{definition}
Let $\cS = (\cA, \cI)$ be an admissible pair of scales and let $\Lambda \subseteq \TO(E,F, \cS)$, $\Lambda^0 \subseteq \TO^0(E,F, \cS)$ be nonempty. An element $R \in \cE(E,F)$ is called \emph{moderate (with respect to $\Lambda$, $\Lambda^0$, and $\cS$)} if 
\begin{gather*}
\forall p \in \csn(F)\ \forall l \in \N\ \forall (\Phi_\e)_\e \in \Lambda \ \forall (\Psi_{1,\e})_\e, \ldots, (\Psi_{l,\e})_\e  \in \Lambda^0\ \exists \lambda \in \cA: \\
p(\ud^l R(\Phi_\e)(\Psi_{1, \e}, \ldots, \Psi_{l, \e}))   = O(\lambda_\e),
\end{gather*}
and \emph{negligible (with respect to $\Lambda$, $\Lambda^0$, and $\cS$)} if 
\begin{gather*}
\forall  p \in \csn(F) \ \forall l \in \N \ \forall (\Phi_\e)_\e \in \Lambda \ \forall (\Psi_{1,\e})_\e, \ldots, (\Psi_{l,\e})_\e  \in \Lambda^0\ \forall \lambda \in \cI: \\
p(\ud^l R(\Phi_\e)(\Psi_{1, \e}, \ldots, \Psi_{l, \e})) = O(\lambda_\e).
\end{gather*}
The set of all moderate (negligible, respectively) elements is denoted by $\cE_\cM(E,F) = \cE_\cM(E,F, \Lambda, \Lambda^0, \cS)$ ($\cE_\cN(E,F) = \cE_\cN(E,F, \Lambda, \Lambda^0, \cS)$, respectively). 
\end{definition}

The following important properties follow immediately from our definitions. In fact, we chose our definitions in such a way precisely for these properties to hold. 

\begin{proposition}\label{properties-embedding}
$\mbox{}$
\begin{enumerate}[label=(\roman*)]
\item $\cE_\cM(E,F)$ is a vector space and $\cE_\cN(E,F)$ is a subspace of $\cE_\cM(E,F)$,
\item $\iota(E) \subseteq \cE_\cM(E,F)$, $\sigma(F) \subseteq \cE_\cM(E,F)$,
\item $\iota(E) \cap \cE_\cN(E,F) = \{0\}$, $\sigma(F) \cap \cE_\cN(E,F) = \{0\}$,
\item $(\iota - \sigma)(F) \subseteq \cE_\cN(E,F)$.
\end{enumerate}
\end{proposition}
We now construct the quotient.
\begin{definition}\label{quotient-global}
Let $\cS = (\cA, \cI)$ be an admissible pair of scales and let $\Lambda \subseteq \TO(E,F, \cS)$, $\Lambda^0 \subseteq \TO^0(E,F, \cS)$ be nonempty. 
The \emph{nonlinear extension of the test pair $(E,F)$ (with respect to $\Lambda$, $\Lambda^0$, and $\cS$)} is defined as 
\[ \cG(E,F) = \cG(E,F, \Lambda, \Lambda^0, \cS) \coleq \cE_\cM(E,F, \Lambda, \Lambda_0, \cS) / \cE_\cN(E,F, \Lambda, \Lambda_0, \cS). \]
The equivalence class of $R \in \cE_{\cM}(E,F)$ is denoted by $[R]$.
\end{definition} 
\Cref{properties-embedding} implies that 
\begin{gather*}
\iota \colon E \to \cG(E,F),\quad \iota(u) \coleq [\iota(u)], \\
\sigma \colon E \to \cG(E,F),\quad \sigma(\varphi) \coleq [\sigma(\varphi)]
\end{gather*}
are linear embeddings such that $\iota|_F = \sigma$. 
The name "nonlinear extension" is justified by the following lemma. 
\begin{lemma}\label{extension-mapping-global}
Let $T \colon F \times \cdots \times F \to F$ be a jointly continuous multilinear mapping and consider the multilinear mapping $\widetilde T \colon \cE(E,F) \times \cdots \times \cE(E,F)\to \cE(E,F)$
given by \eqref{definition1}. Then, $\widetilde T$ preserves moderateness, i.e., $\widetilde T(\cE_{\cM}(E,F), \ldots, \cE_{\cM}(E,F)) \subseteq  \cE_{\cM}(E,F)$, and $\widetilde T(R_1, \ldots, R_n)$ is negligible if at least one of the $R_i$ is negligible. Consequently,
\begin{gather*}
\widetilde T \colon  \cG(E,F) \times \ldots \times \cG(E,F) \to \cG(E,F) \\
\widetilde T([R_1], \ldots, [R_n]) \coleq [T(R_1, \ldots, R_n)]
\end{gather*}
is a well-defined multilinear mapping such that
\[ \widetilde T(\sigma(\varphi_1), \ldots, \sigma(\varphi_n)) = \sigma(T(\varphi_1, \ldots, \varphi_n)). \]
\end{lemma}
\begin{proof}
This follows from \Cref{extension-bs} and the continuity of $T$.
\end{proof}
\begin{corollary}
Suppose that $F$ is a locally convex algebra. Then, $\cE_{M}(E,F)$ is an algebra with multiplication given by \eqref{definition-multiplication} and $\cE_\cN(E,F)$ is an ideal of $\cE_{M}(E,F)$.
Consequently, $\cG(E,F)$ is an algebra with multiplication given by
\[ [R_1] \cdot [R_2] \coleq [R_1 \cdot R_2] \]
and $\sigma$ is an algebra homomorphism.
\end{corollary}
\begin{lemma}\label{diff-global}
Let $(E_1,F_1)$ and $(E_2,F_2)$ be two test pairs. Suppose that $f \colon E_1 \to E_2$ is a linear topological isomorphism such that also the restriction $f|_{F_1}$ is a linear topological isomorphism $F_1 \to F_2$. Let $\cS = (\cA, \cI)$ be an admissible pair of scales and let $\Lambda_i \subseteq \TO(E_i,F_i, \cS)$, $\Lambda_i^0 \subseteq \TO^0(E_i,F_i, \cS)$ be nonempty for $i = 1,2$ such that
\begin{gather*}
(f^{-1} \circ \Phi_\e \circ f)_\e \in \Lambda_1 \qquad \forall (\Phi_\e)_\e \in \Lambda_2, \\
(f^{-1} \circ \Psi_\e \circ f)_\e \in \Lambda^0_1 \qquad \forall (\Psi_\e)_\e \in \Lambda_2^0 \\
\intertext{and}
(f \circ \Phi_\e \circ f^{-1})_\e \in \Lambda_2 \qquad \forall (\Phi_\e)_\e \in \Lambda_1,\\
(f \circ \Psi_\e \circ f^{-1})_\e \in \Lambda^0_2 \qquad \forall (\Psi_\e)_\e \in \Lambda_1^0.
\end{gather*}
Consider the mapping $f_* \colon \cE(E_1,F_1) \to \cE(E_2,F_2)$ given by \eqref{definition2}. Set
\[
\cE_\cM(E_i,F_i) =\cE_\cM(E_i,F_i, \Lambda_i, \Lambda^0_i, \cS), \qquad \cE_\cN(E_i,F_i) =\cE_\cN(E_i,F_i, \Lambda_i, \Lambda^0_i, \cS), \qquad
\]
for $i = 1,2$. Then, $f_*$ preserves moderateness and neglibility. 
Consequently, the mapping $f_* \colon \cG(E_1,F_1) \to \cG(E_2,F_2)$
given by 
\[ f_*([R]) \coleq [f_*(R)] \]  
is an isomorphism that makes the following diagram commutative.
 \[
 \xymatrix{
  E_1 \ar[r]^f \ar[d]_\iota & E_2 \ar[d]^\iota \\
  \cG(E_1, E_2) \ar[r]^f & \cG(E_2, F_2)
  }
 \]
\end{lemma}
\begin{proof}
This follows from \Cref{extension-diff} and the continuity of $f$.
\end{proof}
\begin{lemma}\label{fasdf}
Let $T \in \Linc(E, E)$ with $T|_F \in \Linc(F, F)$. Consider the mapping $\widehat T \colon \cE(E,F) \to \cE(E,F)$ given by \eqref{definition3}. Then, $\widehat T$ preserves moderateness and negligibility. Consequently, the mapping $\widehat{T} \colon \cG(E,F) \to \cG(E,F)$
given by 
\[ \widehat{T}([R]) \coleq [\widehat{T}(R)] \]  
is a well-defined linear mapping that makes the following diagram commutative:
 \[
 \xymatrix{
  E \ar[r]^T \ar[d]_\iota & E \ar[d]^\iota \\
  \cG(E, F) \ar[r]^{\widehat{T}} & \cG(E, F)
  }
 \]
\end{lemma}

\begin{proof}
This follows from \Cref{extension-der} and continuity of $T$.
\end{proof}

\section{Sheaf properties}\label{local-theory}

In this section we study the sheaf theoretic properties of our generalized function spaces. After introducing the necessary terminology, we first look in detail at test objects. Satisfying a certain localizability condition, the spaces of test objects and 0-test objects themselves form sheaves. This is used for showing the existence of global test objects by gluing together local ones, and for extending and restricting test objects in the proof of the sheaf property of the Colombeau quotient.

\subsection{Locally convex sheaves}
Let $X$ be a Hausdorff locally compact paracompact topological space. For open subsets $V,U \subseteq X$ we write $V \Subset U$ to indicate that $\overline{V} \subset U$ and $V$ is relatively compact in $X$. We shall only use this notation for \emph{open} sets.

A \emph{presheaf (of vector spaces) $E$} assigns to each open set $U \subseteq X$ a vector space $E(U)$ and gives, for every inclusion of open sets $V \subseteq U$, a linear mapping $\rho_{V,U} \colon E(U) \to E(V)$ such that for all $W \subseteq V \subseteq U$ the identities $\rho_{W,U} = \rho_{W,V} \circ \rho_{V,U}$ and $\rho_{U,U} = \id$ hold. The elements of $E(U)$ are called sections of $E$ over $U$ and the mappings $\rho_{V,U}$ restriction mappings. 

A presheaf $E$ is a \emph{sheaf} if for all open subsets $U \subseteq X$ and all open coverings $(U_i)_i$ of $U$ the following properties are satisfied: 
\begin{enumerate}[label=(S\arabic*)]
\item\label{S1} If $u \in E(U)$ satisfies $\rho_{U_i, U}(u) = 0$ for all $i$ then $u=0$.
\item\label{S2} If $u_i \in E(U_i)$ are given such that $\rho_{U_i \cap U_j, U_i}(u_i) = \rho_{U_i \cap U_j, U_j}(u_j)$ for all $i,j$ then there exists $u \in E(U)$ such that $\rho_{U_i,U}(u) = u_i$ for all 
$i$.
\end{enumerate}
A section $u \in E(U)$ is said to vanish on an open set $V \subseteq U$ if $\rho_{V,U}(u)= 0$. The support of $u$, denoted by $\supp u$, is defined as the complement in $U$ of the union of all open sets on which $u$ vanishes. The restriction of the sheaf $E$ to an open set $U \subseteq X$ is denoted by $E|_U$.

A \emph{locally convex sheaf $E$} is a sheaf $E$ such that $E(U)$ is a locally convex space for each open set $U \subseteq X$, the restriction mappings are continuous, and for all open sets $U \subseteq X$ and all open coverings $(U_i)_i$ of $U$ the following property is satisfied:
\begin{enumerate}[label=(S\arabic*),resume]
\item\label{S3} the topology on $E(U)$ coincides with the projective topology on $E(U)$ with respect to the mappings $\rho_{U_i, U}$.
\end{enumerate}
Property \ref{S3} and the fact that $X$ is locally compact imply the canonical isomorphism of locally convex spaces
\begin{equation}
E(U) \cong \varprojlim_{W \Subset U} E(W),
\label{S3iso}
\end{equation}
where the projective limit is taken with respect to the restriction mappings. Notice that the algebraic isomorphism in \eqref{S3iso} holds because of \ref{S1} and \ref{S2}. 

Let $E_1$ and $E_2$ be (locally convex) sheaves. A sheaf morphism $\mu\colon E_1 \to E_2$ consists of (continuous) linear mappings $\mu_U\colon E_1(U) \to E_2(U)$ for each open set $U \subseteq X$ such that, for every inclusion of open sets $V \subseteq U$, the identity $\rho_{V,U} \circ \mu_U = \mu_V \circ \rho_{V,U}$ holds. The set of all sheaf morphisms from $E_1$ into $E_2$ is denoted by $\Hom(E_1, E_2)$.  The assignment $U \to \Hom(E_1|_U, E_2|_U)$ together with the canonical restriction mappings is a sheaf. By abuse of notation we shall also denote this sheaf by $\Hom(E_1, E_2)$. More generally, let $E_1, \ldots, E_n,E$ be (locally convex) sheaves on  $X$. A multilinear sheaf morphism $T\colon E_1 \times \cdots \times E_n \to E$ consists of (jointly continuous) multilinear mappings $T_U \colon E_1(U) \times \cdots \times E_n(U) \to E(U)$ for each open set $U \subseteq X$ such that, for every inclusion of open sets $V \subseteq U$, we have
\[ \rho_{V,U} ( T_U(u_1, \ldots, u_n)) = T_V (\rho_{V,U}(u_1), \ldots, \rho_{V,U}(u_n)) \]
if $u_i \in E_i(U)$ for $i = 1, \ldots, n$.

A (locally convex) sheaf $E$ is called a \emph{(locally convex) sheaf of algebras} if for each open set $U \subseteq X$ the space $E(U)$ is a (locally convex) algebra and the multiplication is a bilinear sheaf morphism.

A (locally convex) sheaf $E$ is called \emph{fine} if for all closed subsets $A,B$ of $X$ with $A \cap B = \emptyset$ there is $\mu \in \Hom(E, E)$ and open neighbourhoods $U$ and $V$ of $A$ and $B$, respectively, such that 
$ \mu_U = \id$ and $\mu_V = 0$. Or, equivalently, if  for every open covering $(U_i)_{i}$ of $X$ there is a family $(\eta^{i})_{i} \subset \Hom(E, E)$ such that the family of supports of the $\eta^i$ is locally finite, $\supp \eta^{i} \subseteq U_i$ for all $i$, and $\sum_{i} \eta^{i} = \id$. The family $(\eta^{i})_{i}$ is called a partition of unity subordinate to the covering $(U_i)_i$. We shall often use the following extension principle for (locally convex) fine sheaves $E$: Let $U,V,W$ be open subsets of $X$ such that $\overline{W} \subset V \subseteq U$. Then, there is a (continuous) linear mapping $\tau \colon E(V) \to E(U)$ such that $\rho_{W,V} = \rho_{W, U} \circ \tau$. 

\subsection{Localizing regularization operators} Let $X$ be a Hausdorff locally compact paracompact topological space and $E$ and $F$ locally convex sheaves. We call $(E,F)$ a \emph{test pair of sheaves} if the following three properties are satisfied:
\begin{enumerate}[label=(\roman*)]
\item\label{TP1} $F$ is a subsheaf of $E$.
\item\label{TP2} $(E(U),F(U))$ is a test pair for each open set $U \subseteq X$. 
\end{enumerate}
Given a sheaf morphism $\mu \in \Hom(E,E)$ we write $\mu|_F$ for its restriction to $F$. Hence $\mu|_F \in \Hom(F,F)$ means that $\mu_U|_{F(U)}$ is a continuous linear operator from $F(U)$ into itself for each open set $U \subseteq X$.  The third property can then be formulated as follows: 
\begin{enumerate}[label=(\roman*),resume]
\item\label{TP3} For all open sets $U \subseteq X$ and all closed subsets $A,B$ of $U$ with $A \cap B = \emptyset$ there is $\mu \in \Hom(E|_U, E|_U)$ with $\mu|_F$ $\in \Hom(F|_U,F|_U)$ such that 
$ \mu_V = \id$ and $\mu_W = 0$ for some open neighbourhoods $V$ and $W$ (in $U$) of $A$ and $B$, respectively. Or, equivalently, to the fact that for any open set $U$ of $X$ and any open covering $(U_i)_{i}$ of $U$ there is a partition of unity $(\eta^{i})_{i} \subset \Hom(E|_U, E|_U)$ subordinate to $(U_i)_i$ such that $\eta^{i}|_{F|_U} \in \Hom(F|_U, F|_U)$ for all $i$.
\end{enumerate}
In particular, property \ref{TP3} implies that $E|_U$ and $F|_U$ are fine sheaves for all open sets $U \subseteq X$.  Moreover, it implies that for all  open subsets $U,V,W$ of $X$ with $\overline{W} \subset V \subseteq U$ there is $\tau \in \Linc(E(V), E(U))$ such that $\rho_{W,V} = \rho_{W, U} \circ \tau$ and $\tau|_{F(V)} \in \Linc(F(V), F(U))$ .

Since $F$ is a subsheaf of $E$, there is no need to make a distinction between the restriction mappings on $E$ and $F$, respectively. These mappings will be denoted by $\rho_{U,V}$.
Furthermore, we introduce the shorthand notation $\RO(U) = \Linc(E(U),F(U))$, where $\RO$ stands for ``regularization operator''.

\begin{definition}
Let $U \subseteq X$ be open. An element $(\Phi_\e)_\e \in \RO(U)^I$ is called \emph{localizing} if 
\begin{gather*}
(\forall V, V_0 \subseteq X: V \Subset V_0 \Subset U)\ (\exists \e_0 \in I)\ (\forall \e<\e_0)\ (\forall u \in E(U)) \\
(\rho_{V_0, U}(u) = 0 \Rightarrow \rho_{V, U}(\Phi_\e(u)) = 0).
\end{gather*}
We write $\ROloc(U)$ for the set of all localizing elements in $\RO(U)^I$. Furthermore, we define 
\begin{gather*}
\TOloc(U) = \TOloc(U, \cS) \coleq \TO(E(U),F(U), \cS) \cap \ROloc(U) \\
\TOzloc(U) = \TOzloc(U, \cS) \coleq \TO^0(E(U),F(U), \cS) \cap \ROloc(U),
\end{gather*}
where $\cS$ is an admissible pair of scales. 
\end{definition}
\begin{remark}
Throughout this subsection we shall always assume that the space $\TOloc(U)$ is nonempty.  
\end{remark}
\begin{definition}
Let $U \subseteq X$ be open. We define $\NO(U)$ as the vector space consisting of all $(\Phi_\e)_\e \in \RO(U)^I$ such that for all $V \Subset U$ we have $\rho_{V,U} \circ \Phi_\e = 0$ for $\e$ small enough. Define 
\[
\tROloc(U) \coleq \ROloc(U) \slash \NO(U), \qquad \tTOzloc(U) \coleq \TOzloc(U) \slash \NO(U).
\]
 For $(\Phi_\e)_\e, (\Phi'_\e)_\e \in \RO(U)^I$ we write $(\Phi_\e)_\e \sim (\Phi'_\e)_\e$ if $(\Phi_\e - \Phi'_\e)_\e \in \NO(U)$.
Set \[ \tTOloc(U) \coleq \TOloc(U) \slash{\sim}. \]
\end{definition}
The main goal of this section is to show that one can define a natural sheaf structure on $U \to \tROloc(U)$. We start with defining the restriction mappings.

\begin{lemma}\label{restriction}
Let $U,V$ be open subsets of $X$ with $V \subseteq U$. There is a linear mapping $\rho^{\RO}_{V,U}\colon \RO(U) \to \RO(V)$ which is continuous for the strong topologies on $\RO(U)$ and  $\RO(V)$ and such that for all $ (\Phi_\e)_\e \in \ROloc(U)$ the following properties hold:
\begin{enumerate}[label=(\roman*)]
\item\label{5.4.1} We have that
\begin{gather*}
(\forall W, W_0 \subseteq X: W \Subset W_0 \Subset V)\ (\exists \e_0 \in I)\ (\forall \e<\e_0)\ (\forall u \in E(U))\ (\forall v \in E(V)) \\
(\rho_{W_0, U}(u) = \rho_{W_0, V}(v) \Rightarrow \rho_{W, V}(\rho^{\RO}_{V,U}(\Phi_\e)(v)) = \rho_{W, U}(\Phi_\e(u))).
\end{gather*}
\item\label{5.4.2} For all $W \Subset V$ and all $\tau \in \Linc(E(V),E(U))$ with $\rho_{W_0,U} \circ \tau = \rho_{W_0,V}$ for some $W \Subset W_0 \Subset V$ we have that
\[ \rho_{W,U} \circ \Phi_\e \circ \tau = \rho_{W,V} \circ \rho^{\RO}_{V,U}(\Phi_\e) \]
for $\e$ small enough.
\item\label{5.4.3} For all $W \Subset V$ we have that
\[ \rho_{W,V} \circ \rho^{\RO}_{V,U}(\Phi_\e) \circ \rho_{V,U} = \rho_{W,U} \circ \Phi_\e \]
for $\e$ small enough.
\item\label{5.4.4} For all $W \Subset V$ and $\Phi_1, \Phi_2 \in \ROloc(U)$ that satisfy
\[ \rho_{W,U} \circ \Phi_1 = \rho_{W,U} \circ \Phi_2 \]
we have that \[ \rho_{W,V} \circ \rho^{\RO}_{V,U}(\Phi_1) = \rho_{W,V} \circ \rho^{\RO}_{V,U}(\Phi_2). \]
\end{enumerate}
\end{lemma}
\begin{proof}
Let $(V_i)_{i}$ be an open covering of $V$ such that $V_i \Subset V$ for all $i$. Let $(\eta^{i})_i \subset \Hom(F|_V, F|_V)$ be a partition of unity subordinate to $(V_i)_{i}$ and choose $\tau_i \in \Linc(E(V), E(U))$ such that 
$\rho_{V_i,V} = \rho_{V_i, U} \circ \tau_i$ for all $i$. We define
\[ \rho^{\RO}_{V,U}(\Phi) \coleq \sum_{i} \eta^{i}_V \circ \rho_{V,U} \circ \Phi \circ \tau_i. \]
For all $W \Subset V$ it holds that $\supp \eta^i \cap W = \emptyset$ except for $i$ belonging to some finite index set $J$. Hence
\begin{equation}
\rho_{W,V} \circ \rho^{\RO}_{V,U}(\Phi) = \sum_{i \in J} \eta^{i}_W \circ \rho_{W,U} \circ \Phi \circ \tau_i,
\label{representation-compacts-0}
\end{equation}
By \eqref{S3iso} we then have that $\rho^{\RO}_{V,U}(\Phi) \in \RO(V)$. The linearity and continuity of $\rho^{\RO}_{V,U}$ and also \ref{5.4.4} are clear from this expression. We now show \ref{5.4.1}. Let $W 
\Subset V$ and $W \Subset W_0 \Subset V$ be arbitrary. Suppose that the representation \eqref{representation-compacts-0} holds for some finite index set $J$. 
Choose $V'_{i} \Subset V_i$ such that $\supp \eta^{i} \subset V'_i$.
Since $(\Phi_\e)_\e$ is localizing, there is  $\e_0 \in I$ such that for all $i \in J$, $\e < \e_0$, and $u \in E(U)$ it holds that
\begin{equation}
\rho_{W_0 \cap V_i, U}(u) = 0 \Rightarrow \rho_{W \cap V'_i, U}(\Phi_\e(u)) = 0. 
\label{localizing-in-proof}
\end{equation}
Assume that $u \in E(U)$ and $v \in E(V)$ are given such that $\rho_{W_0, U}(u) = \rho_{W_0, V}(v)$. Since 
\[ \rho_{W, U} \circ \Phi_\e = \sum_{i \in J} \eta^{i}_{W} \circ \rho_{W,U} \circ \Phi_\e \]
and $\supp \eta^i  \subset V'_i$ it suffices to show that
\[ \rho_{W \cap V'_i, U}(\Phi_\e(u - \tau_i(v))) = 0 \]
for all $i \in J$. This follows from \eqref{localizing-in-proof} and our choice of $\tau_i$.
Properties \ref{5.4.2} and \ref{5.4.3} are special cases of \ref{5.4.1}.
\end{proof}

\begin{lemma} \label{presheaf-RO}
Let $U,V$ be open subsets of $X$ with $V \subseteq U$. Then, for all $ (\Phi_\e)_\e \in \ROloc(U)$ it holds that 
\begin{enumerate}[label=(\roman*)]
\item\label{5.5.1} $(\rho^{\RO}_{V,U}(\Phi_\e))_\e \in \ROloc(V)$,
\item\label{5.5.2} if $(\Phi_\e)_\e \sim 0$, then $(\rho^{\RO}_{V,U}(\Phi_\e))_\e \sim 0$,
\item\label{5.5.3} for $W \subseteq V \subseteq U$ it holds that $((\rho^{\RO}_{W,V} \circ\rho^{\RO}_{V,U})(\Phi_\e))_\e \sim (\rho^{\RO}_{W,U}(\Phi_\e))_\e$.
\end{enumerate}
\end{lemma}
\begin{proof}
\ref{5.5.1} Let $W \Subset V$ and $W \Subset W_0 \Subset V$ be arbitrary. Since $ (\Phi_\e)_\e$ is localizing there is $\e_1 \in I$ such that such that for all $\e < \e_1$ and all $u \in E(U)$ it holds that
\begin{equation}
\rho_{W_0, U}(u) = 0 \Rightarrow \rho_{W , U}(\Phi_\e(u)) = 0. 
\label{localizing-in-proof-3}
\end{equation}
Choose $\tau \in \Linc(E(V), E(U))$ such that $\rho_{W_0,U} \circ \tau = \rho_{W_0,V}$. By \Cref{restriction} \ref{5.4.2} there is $\e_2 \in I$ such that  
\[ \rho_{W,V} \circ \rho^{\RO}_{V,U}(\Phi_\e) = \rho_{W,U} \circ \Phi_\e \circ \tau \]
for all $\e < \e_2$. Set $\e_0 = \min(\e_1, \e_2)$. Let $v \in E(V)$ be such that $\rho_{W_0,V}(v) = 0$. Hence
\[ \rho_{W,V}(\rho^{\RO}_{V,U}(\Phi_\e)(v)) = \rho_{W,U}(\Phi_\e(\tau(v))) = 0 \]
for all $\e < \e_0$.
 
\ref{5.5.2} Let $W \Subset V$ be arbitrary.  Choose $\tau \in \Linc(E(V), E(U))$ such that $\rho_{W_0,U} \circ \tau = \rho_{W_0,V}$. By \Cref{restriction} \ref{5.4.2} we have that 
\[ \rho_{W,V} \circ \rho^{\RO}_{V,U}(\Phi_\e)= \rho_{W,U} \circ \Phi_\e \circ \tau = 0 \]
for $\e$ small enough because $(\Phi_\e)_\e \sim 0$.

\ref{5.5.3} Let $W_0 \Subset W$ be arbitrary. Fix an open set $W'_0$ such that $W_0 \Subset W'_0 \Subset W$. Choose $\tau \in \Linc(E(V), E(U))$ such that $\rho_{W'_0,U} \circ \tau = \rho_{W'_0,V}$ and $\tau' \in \Linc(E(W), E(V))$ such that $\rho_{W'_0,V} \circ \tau = \rho_{W'_0,W}$. Hence $\tau \circ \tau' \in \Linc(E(W), E(U))$ and $\rho_{W'_0,U} \circ \tau \circ \tau' = \rho_{W'_0,W}$  By \Cref{restriction} \ref{5.4.2} we have that 
\begin{gather*}
\rho_{W_0,W} \circ \rho^{\RO}_{W,V}(\rho^{\RO}_{V,U}(\Phi_\e)) = \rho_{W_0,V} \circ \rho^{\RO}_{V,U}(\Phi_\e) \circ \tau \\
= \rho_{W_0,U} \circ \Phi_\e \circ\ \tau' \circ \tau = \rho_{W_0,W} \circ \rho^{\RO}_{W,U}(\Phi_\e) 
\end{gather*}
for $\e$ small enough.
\end{proof}
\Cref{presheaf-RO} implies that the mappings 
\[ \rho^{\RO}_{V,U}([(\Phi_\e)_\e]) \coleq [(\rho^{\RO}_{V,U}(\Phi_\e))_\e] \]
define a presheaf structure on $U \to \tROloc(U)$. We now show that it is in fact a sheaf.
\begin{proposition}\label{sheaf-RO}
$\tROloc$ is a  sheaf of vector spaces. 
\end{proposition}
\begin{proof}
Let $U \subseteq X$ be open and let $(U_i)_i$ be an open covering of $U$.

\ref{S1} Suppose that $[(\Phi_\e)_\e] \in \tROloc(U)$ such that 
\[ \rho^{\RO}_{U_i,U}([(\Phi_\e)_\e]) = 0 \] 
for all $i$. We need to show that $(\Phi_\e)_\e \sim 0$. Let $W \Subset U$ be arbitrary. We may assume without loss of generality that $W \Subset U_i$ for some $i$. By \Cref{restriction} \ref{5.4.3} and our assumption we have that 
\[ \rho_{W,U} \circ \Phi_\e = \rho_{W,U_i} \circ \rho^{\RO}_{U_i,U}(\Phi_\e) \circ \rho_{U_i,U} = 0 \]
for $\e$ small enough.

\ref{S2} Since $X$ is locally compact we may assume without loss of generality that $U_i \Subset U$ for all $i$. Suppose that $[(\Phi_{i,\e})_\e] \in \tROloc(U_i)$ are given such that  
\[ \rho^{\RO}_{U_i \cap U_j,U_i}([(\Phi_{i,\e})_\e]) = \rho^{\RO}_{U_i \cap U_j,U_j}([(\Phi_{j,\e})_\e]) \]
for all $i,j$. Let $(\eta^{i})_i \subset \Hom(F|_U, F|_U)$ be a partition of unity subordinate to the covering $(U_i)_i$. Choose $\tau_i \in \Linc(F(U_i), F(U))$ such that $\rho_{V_i, U} \circ \tau_i = \rho_{V_i,U_i}$ for some $V_i \Subset U_i$ with $\supp \eta^{i} \subset V_i$. We define
\[ \Phi_\e = \sum_i \eta^{i}_U \circ \tau_i \circ \Phi_{i,\e} \circ \rho_{U_i, U} \]
for all $\e \in I$. Notice that $\Phi_\e \in \RO(U)$ because of \eqref{S3iso} and the fact that the family of supports of the $\eta^{i}$ is locally finite. We now show that $(\Phi_\e)_\e$ is localizing. Let $W \Subset U$ and $W \Subset W_0 \Subset U$ be arbitrary and suppose that $\supp \eta^{i} \cap W = \emptyset$ except for $i$ belonging to some finite index set $J$. Choose $V_i' \Subset U_i$ such that $V_i \Subset V'_i$. Since the $(\Phi_{i,\e})_\e$ are localizing there is $\e_0 \in I$ such that for all $i \in J$, $\e < \e_0$, and $u \in E(U_i)$ it holds that 
\begin{equation}
\rho_{W_0 \cap V'_i, U_i}(u) = 0 \Rightarrow \rho_{W \cap V_i, U_i}(\Phi_{i,\e}(u)) = 0. 
\label{localizing-in-proof-2}
\end{equation}
Now suppose that $u \in E(U)$ satisfies $\rho_{W_0, U}(u) = 0$. Since
\[ \rho_{W, U}(\Phi_\e(u)) = \sum_{i \in J} \eta^{i}_{W} ( \rho_{W, U} (\tau_i(\Phi_{i,\e}(\rho_{U_i, U}(u))))) \] 
and $\supp \eta^{i} \subset V_i$ it suffices to show that
\[ \rho_{W \cap V_i, U} (\tau_i(\Phi_{i,\e}(\rho_{U_i, U}(u))))= \rho_{W \cap V_i, U_i}(\Phi_{i,\e}(\rho_{U_i, U}(u))) = 0 \]
for all $i \in J$. This follows from \eqref{localizing-in-proof-2}.
Finally, we show that
\[ \rho^{\RO}_{U_i,U}([(\Phi_\e)_\e]) = [(\Phi_{i,\e})_\e] \]
for all $i$. Let $W \Subset U_i$ be arbitrary and suppose that $\supp \eta^{j} \cap W = \emptyset$ except for $j$ belonging to some finite index set $J$.  Let $\tau \in \Linc(E(U_i), E(U))$ be such that $\rho_{W_0, U} \circ \tau = \rho_{W_0,U_i}$ where $W_0$ is some open set such that $W \Subset W_0 \Subset U_i$. \Cref{restriction} \ref{5.4.2} yields that
\begin{gather*}
\rho_{W,U_i} \circ \rho^{\RO}_{U_i,U}(\Phi_\e) - \rho_{W,U_i} \circ \Phi_{i,\e} = \rho_{W,U} \circ \Phi_\e \circ \tau - \rho_{W,U_i} \circ \Phi_{i,\e} \\ 
= \sum_{j \in J} \eta^{j}_W \circ ( \rho_{W,U} \circ \tau_j \circ \Phi_{j,\e} \circ \rho_{U_j, U} \circ \tau - \rho_{W,U_i} \circ \Phi_{i,\e}).
\end{gather*}
Since $\supp \eta^{j} \subset V_j$ it suffices to show that 
\[ \rho_{W \cap V_j,U} \circ \tau_j \circ \Phi_{j,\e} \circ \rho_{U_j, U} \circ \tau - \rho_{W \cap V_j,U_i} \circ \Phi_{i,\e} = 0 \]
for all $j \in J$ and $\e$ small enough. Our choice of $\tau_j$ and \Cref{restriction} \ref{5.4.2} and \ref{5.4.3} imply that
\begin{align*}
&\rho_{W \cap V_j,U} \circ \tau_j \circ \Phi_{j,\e} \circ \rho_{U_j, U} \circ \tau - \rho_{W \cap V_j,U_i} \circ \Phi_{i,\e} \\
&= \rho_{W \cap V_j,U_j} \circ \Phi_{j,\e} \circ \rho_{U_j, U} \circ \tau - \rho_{W \cap V_j,U_i} \circ \Phi_{i,\e} \\
&=  \rho_{W \cap V_j, U_i \cap U_j} \circ \rho^{\RO}_{U_i \cap U_j, U_j}(\Phi_{j,\e}) \circ \rho_{U_i \cap U_j, U}  \circ \tau - \rho_{W \cap V_j,U_i} \circ \Phi_{i,\e} \\ 
&=  \rho_{W \cap V_j, U_i \cap U_j} \circ \rho^{\RO}_{U_i \cap U_j, U_i}(\Phi_{i,\e}) \circ \rho_{U_i \cap U_j, U}  \circ \tau - \rho_{W \cap V_j,U_i} \circ \Phi_{i,\e} \\ 
&=  \rho_{W \cap V_j, U_i} \circ \Phi_{i,\e} \circ \rho_{U_i, U}  \circ \tau - \rho_{W \cap V_j,U_i} \circ \Phi_{i,\e}
\end{align*}
which equals zero for $\e$ small enough because $(\Phi_{i,\e})_\e$ is localizing.
\end{proof}

\begin{lemma}\label{morphisms-RO}
Every sheaf morphism $\mu \in \Hom(F,F)$ induces a sheaf morphism $\mu \in \Hom(\tROloc, \tROloc)$ via 
\begin{equation}
\mu_U([(\Phi_\e)_\e]) \coleq [(\mu_U  \circ \Phi_\e)_\e],   
\label{sheaf-morphism-extension}
\end{equation}
with $U$ an open subset of $X$.
\end{lemma}
\begin{proof}
Clearly, $\mu_U\colon \tROloc(U) \to \tROloc(U)$ is a well-defined linear mapping for each $U \subseteq X$ open. We now show that $\mu$ is a sheaf morphism. Let $V,U$ be open subsets of $X$ such that $V \subseteq U$. It suffices to show that for all $W \Subset V$ and all $(\Phi_\e)_\e \in \ROloc(U)$ it holds that
\[ \rho_{W,V} \circ \rho^{\RO}_{V,U}(\mu_U \circ \Phi_\e) = \rho_{W,V} \circ \mu_V \circ \rho^{\RO}_{V,U}(\Phi_\e) \]
for $\e$ small enough. Let $\tau \in \Linc(E(V),E(U))$ be such that $\rho_{W_0, U} \circ \tau = \rho_{W_0,V}$ for some open set $W_0$ such that $W \Subset W_0 \Subset V$. By \Cref{restriction} \ref{5.4.2} we have that
\begin{align*}
\rho_{W,V} \circ \rho^{\RO}_{V,U}(\mu_U \circ \Phi_\e)  &= \rho_{W,U} \circ \mu_U \circ \Phi_\e \circ \tau \\ 
& = \mu_W \circ \rho_{W,U} \circ \Phi_\e \circ \tau \\
& = \mu_W \circ \rho_{W,V} \circ \rho^{\RO}_{V,U}(\Phi_\e) \\
& = \rho_{W,V} \circ \mu_V \circ \rho^{\RO}_{V,U}(\Phi_\e) \\
\end{align*}
for $\e$ small enough.
\end{proof}
We now turn our attention to spaces of test objects.
\begin{lemma} \label{sheaf-TO} 
Let $U \subseteq X$ be open and let $(U_i)_i$ be an open covering of $U$. Let $(\Phi_\e)_\e \in  \ROloc(U)$. Then, $(\Phi_\e)_\e \in  \TOloc(U)$ ($(\Phi_\e)_\e \in  \TOzloc(U)$, respectively) if and only if $(\rho^{\RO}_{U_i,U}(\Phi_\e))_\e \in \TOloc(U_i)$ ($(\rho^{\RO}_{U_i,U}(\Phi_\e))_\e \in \TOzloc(U_i)$, respectively) for all $i$. 
\end{lemma}
\begin{proof}
We only show the statement for $\TOloc$, the proof for $\TOzloc$ is similar. Let  $(\Phi_\e)_\e \in \ROloc(U)$. We first assume that $(\Phi_\e)_\e$ satisfies $(\TO)_j$ with $j = 1,2$ or $3$, and prove that $(\rho^{\RO}_{U_i, U}(\Phi_\e))_\e$ does so as well.

$j = 1$: It suffices to show that for all $u \in E(U_i)$ and all $p \in \csn(F(W))$, with $W \Subset U_i$ arbitrary, there is $\lambda \in \cA$ such that
\[ p(\rho_{W,U_i}(\rho^{\RO}_{U_i,U}(\Phi_\e)(u)) = O(\lambda_\e). \]
Let $\tau \in \Linc(E(U_i),E(U))$ such that $ \rho_{W_0,U} \circ \tau = \rho_{W_0,U_i}$ where $W_0$ is an open set such that $W \Subset  W_0 \Subset U_i$. By  \Cref{restriction} \ref{5.4.2} we have that 
\[ \rho_{W,U_i}(\rho^{\RO}_{U_i,U}(\Phi_\e)(u)) = \rho_{W,U}(\Phi_\e(\tau(u))) \]
for $\e$ small enough. The result now follows from our assumption and the fact that $\rho_{W,U} \in \Linc(F(U),F(V))$.

$j = 2$: It suffices to show that for all $\varphi \in F(U_i)$, all $p \in \csn(F(W))$, with $W \Subset U_i$ arbitrary,  and all $\lambda \in \cI$ it holds that
\[ p(\rho_{W,U_i}(\rho^{\RO}_{U_i,U}(\Phi_\e)(\varphi) - \varphi)) = O(\lambda_\e). \]
Let $\tau \in \Linc(F(U_i),F(U))$ such that $ \rho_{W_0,U} \circ \tau = \rho_{W_0,U_i}$ where $W_0$ is an open set such that $W \Subset  W_0 \Subset U_i$. By \Cref{restriction} \ref{5.4.2} we have that 
\[ \rho_{W,U_i}(\rho^{\RO}_{U_i,U}(\Phi_\e)(\varphi) - \varphi) = \rho_{W,U}(\Phi_\e(\tau(\varphi)) - \tau(\varphi)) \]
for $\e$ small enough. The result now follows from our assumption and the fact that $\rho_{W,U} \in \Linc(F(U), F(W))$.

$j = 3$: It suffices to show that for all $u \in E(U_i)$ and all $p \in \csn(E(W))$, with $W \Subset U_i$ arbitrary, it holds that
\[ p(\rho_{W,U_i}(\rho^{\RO}_{U_i,U}(\Phi_\e)(u) - u))  \to 0. \]
Let $\tau \in \Linc(E(U_i),E(U))$ such that $ \rho_{W_0,U} \circ \tau = \rho_{W_0,U_i}$ where $W_0$ is an open set such that $W \Subset  W_0 \Subset U_i$. By \Cref{restriction} \ref{5.4.2} we have that
\[ \rho_{W,U_i}(\rho^{\RO}_{U_i,U}(\Phi_\e)(u) - u) = \rho_{W,U}(\Phi_\e(\tau(u)) - \tau(u)) \]
for $\e$ small enough. The result now follows from our assumption and the fact that $\rho_{W,U} \in \Linc(E(U),E(V))$.

Conversely, assume that $(\rho^{\RO}_{U_i,U}(\Phi_\e))_\e$ satisfies $\TO_j$ with $j = 1,2$ or $3$ for each $i$. We will prove that $(\Phi_\e)_\e$ does so as well.

$j = 1$: 
It suffices to show that for all $u \in E(U)$ and all $p \in \csn(F(W))$, with $W \Subset U_i$ (for some $i$) arbitrary, there is $\lambda \in \cA$ such that
\[ p(\rho_{W,U}((\Phi_\e(u)))) = O(\lambda_\e). \]
By \Cref{restriction} \ref{5.4.3} we have that
\[ \rho_{W,U}(\Phi_\e)(u) = \rho_{W,U_i}(\rho^{\RO}_{U_i,U}(\Phi_\e)(\rho_{U_i,U}(u))) \]
for $\e$ small enough and the result follows from our assumption and the fact that $ \rho_{W,U_i} \in \Linc(F(U_i), F(W))$.

$j = 2$: It suffices to show that for all $\varphi \in F(U)$ and all $p \in \csn(F(W))$, with $W \Subset U_i$ (for some $i$) arbitrary, and all $\lambda \in \cI$ it holds that
\[ p(\rho_{W,U}((\Phi_\e)(\varphi) - \varphi)) = O(\lambda_\e). \]
By \Cref{restriction} \ref{5.4.3} we have that
\[ \rho_{W,U}((\Phi_\e)(\varphi) - \varphi) = \rho_{W,U_i}(\rho^{\RO}_{U_i,U}(\Phi_\e)(\rho_{U_i,U}(\varphi))- \rho_{U_i,U}(\varphi)) \]
for $\e$ small enough and the result follows from our assumption and the fact that $\rho_{W,U_i} \in \Linc(F(U_i), F(W))$.

$j = 3$: It suffices to show that for all $u \in E(U)$ and all $p \in \csn(E(W))$, with $W \Subset U$ arbitrary, it holds that
\[ p((\Phi_\e)(u) - u)  \to 0. \]
By \Cref{restriction} \ref{5.4.3} we have that
\[ \rho_{W,U}((\Phi_\e)(u) - u) = \rho_{W,U_i}(\rho^{\RO}_{U_i,U}(\Phi_\e)(\rho_{U_i,U}(u))- \rho_{U_i,U}(u)) \]
for $\e$ small enough and the result follows from our assumption and the fact that $\rho_{W,U_i} \in \Linc(E(U_i), E(W))$.
\end{proof}
The following result is an immediate consequence of \Cref{mappings-TO} and \Cref{morphisms-RO}.
\begin{lemma} \label{morphisms-TO}  $\mbox{}$
\begin{enumerate}[label=(\roman*)]
\item Let $\mu^i \in \Hom(E, E)$, $i= 0, \ldots, N$, $N \in \N$, be such that $\mu^i|_F  \in \Hom(F,F)$ and $\sum_{i = 0}^N \mu^i = \id$. Then, $\left( \sum_{i = 0}^N \mu^{i}_U \circ \Phi_{i,\e}\right)_\e \in \TOloc(U)$ for all $(\Phi_{i,\e})_\e \in \TOloc(U)$, $i= 0, \ldots, N$.
\item Let $\mu \in \Hom(E, E)$, be such that $\mu|_F \in \Hom(F,F)$. Then, $(\mu_U \circ \Phi_{\e})_\e \in \TOzloc(U)$ for all $(\Phi_{\e})_\e \in \TOzloc(U)$.
\item Let $\mu \in \Hom(E,E)$ be such that $\mu|_F \in \Hom(F,F)$. Then, $(\mu_U \circ \Phi_\e - \Phi_\e \circ \mu_U)_\e \in \TOzloc(U)$ for all $(\Phi_\e)_\e \in \TOloc(U) \cup \TOzloc(U)$.
\end{enumerate}
\end{lemma}
We conclude this subsection with a lemma that will be very useful later on.
\begin{lemma}\label{soft-RO} Let $W, V, U$ be open sets in $X$ such that $W \Subset V \subseteq U$. For every $(\Phi_\e)_ \e \in  \ROloc(V)$ ($(\Phi_\e)_ \e \in  \TOloc(V)$,$(\Phi_\e)_ \e \in  \TOzloc(V)$ respectively) there is $(\Phi'_\e)_ \e \in  \ROloc(U)$ ($(\Phi'_\e)_ \e \in  \TOloc(U)$,$(\Phi'_\e)_ \e \in  \TOzloc(U)$ respectively) such that
\[ \rho_{W, V} \circ \Phi_\e \circ \rho_{V,U} =\rho_{W, U} \circ \Phi'_\e \]
for $\e$ small enough.
\end{lemma}
\begin{proof}
We only show the statement for $(\Phi_\e)_ \e \in  \TOloc(V)$, the other cases can be treated similarly. Choose open sets $W_0, W_1$ such that $W \Subset W_0 \Subset W_1 \Subset V$ and let $\mu \in \Hom(E,E)$ be such that $\mu|_F \in \Hom(F,F)$, $\mu_{W_0} = \id$, and $\mu_{U \backslash \overline{W_1}} = 0$. Furthermore, pick an arbitrary element $(\Phi''_\e)_\e \in  \TOloc(U)$. By  \Cref{sheaf-TO} we have that
\[ \rho^{\RO}_{U \backslash \overline{W_1},U}([(\Phi''_\e)_\e]) \in  \tTOloc(U \backslash \overline{W_1}), \]
and by \Cref{morphisms-TO} it holds that
\[ \mu_V([(\Phi_\e)_\e]) + (\id -\mu)_V (\rho^{\RO}_{V,U}([(\Phi''_\e)_\e])) \in  \tTOloc(V). \]
Since
\[
\rho^{\RO}_{U \backslash \overline{W_1} \cap V,U \backslash \overline{W_1}}(\rho^{\RO}_{U \backslash \overline{W_1},U}([(\Phi''_\e)_\e])) =  \rho^{\RO}_{U \backslash \overline{W_1} \cap V,V}(\mu_V([(\Phi_\e)_\e]) + (\id -\mu)_V (\rho^{\RO}_{V,U}([(\Phi''_\e)_\e]))),
\]
\Cref{sheaf-RO} and \Cref{sheaf-TO} imply that there is an element $(\Phi'_\e)_\e \in  \TOloc(U)$ such that
\[ \rho^{\RO}_{W_0,U}([(\Phi'_\e)_\e]) = \rho^{\RO}_{W_0,V}(\mu_V([(\Phi_\e)_\e]) + (\id -\mu)_V (\rho^{\RO}_{V,U}([(\Phi''_\e)_\e]))) = \rho^{\RO}_{W_0,V}([(\Phi_\e)_\e]). \]
The result now follows from \Cref{restriction} \ref{5.4.3}.
\end{proof}

\subsection{Sheaves of nonlinear extensions}Let $(E,F)$ be a test pair of sheaves. 
We write $\cE(U) = \cE(E(U),F(U))$.
\begin{definition}\label{basicspaces}
$R \in \cE(U)$ is called \emph{local} if for all $V\subseteq U$ and all $\Phi_1, \Phi_2 \in \RO(U)$ the implication
\[ ( \rho_{V, U} \circ \Phi_1 = \rho_{V, U} \circ \Phi_2 ) \Longrightarrow (\rho_{V,U} ( R(\Phi_1)) = \rho_{V, U} ( R(\Phi_2)) ) \]
holds. The set of all local elements of $\cE(U)$ is denoted by $\cEloc(U)$.
\end{definition}
\begin{remark}\label{local-derivatives}
If $R \in \cE(U)$ is local then the identities 
\[ \rho_{V, U} \circ \Phi_1 = \rho_{V, U} \circ \Phi_2, \qquad \rho_{V, U} \circ \Psi_{i,1} = \rho_{V, U} \circ \Psi_{i,2}, \]
with $\Phi_1, \Phi_2 , \Psi_{1,i}, \Psi_{2,i} \in \RO(U)$ for $i = 1, \ldots, l$ imply that
\[ \rho_{V,U} ( \ud^l R(\Phi_1)(\Psi_{1,1}, \ldots, \Psi_{l,1})) =\rho_{V,U} ((\ud^lR)(\Phi_2)(\Psi_{1,2}, \ldots, \Psi_{l,2})). \]
\end{remark}
Next, we define a restriction mapping on $\cEloc$.
\begin{lemma}\label{restriction-basic-space}
Let $U,V$ be open subsets of $X$ with $V \subseteq U$. There is a unique linear mapping $\rho^{\cE}_{V,U}\colon  \cEloc(U) \to \cEloc(V)$ such that 
\begin{enumerate}[label=(\roman*)]
\item\label{5.13.1} for all  $W \Subset V$, $\Phi \in \RO(V)$, and $\Phi' \in \RO(U)$ which satisfy
\[
\rho_{W, V} \circ \Phi \circ \rho_{V,U} =\rho_{W, U} \circ \Phi', 
\] we have
\[ \rho_{W, V} (\rho^{\cE}_{V,U}(R)(\Phi)) = \rho_{W, U} (R(\Phi')). \]
\end{enumerate}
Moreover, the following properties are satisfied:
\begin{enumerate}[label=(\roman*),resume]
\item\label{5.13.2} For all $l \in \N$ and all $W \Subset V$ it holds that if  $\Phi \in \RO(V)$, $\Phi' \in \RO(U)$ and  $\Psi_i \in \RO(V)$, $\Psi'_i \in \RO(U)$, $i = 0, \ldots, l$,  satisfy
\[ \rho_{W, V} \circ \Phi \circ \rho_{V,U} =\rho_{W, U} \circ \Phi',  \qquad \rho_{W, V} \circ \Psi_i \circ \rho_{V,U} =\rho_{W, U} \circ \Psi'_i \]
for all $i = 1, \ldots, l$, then
\[ \rho_{W, V}((\ud^l(\rho^{\cE}_{V,U}(R)))(\Phi)(\Psi_1, \ldots, \Psi_l))= \rho_{W, U}((\ud^lR)(\Phi')(\Psi'_1, \ldots, \Psi'_l)). \]
\item\label{5.13.3} For $W \subseteq V \subseteq U$ it holds that $\rho^{\cE}_{W,V} \circ \rho^{\cE}_{V,U} = \rho^{\cE}_{W,U}$.
\end{enumerate}
\end{lemma}
\begin{proof}
Let $(V_i)_i$ be an open covering of $V$ such that $V_i \Subset V$ for all $i$ and let $(\eta^{i})_i \subset \Hom(F|_V,F|_V)$ be a partition of unity subordinate to $(V_i)_i$. Choose $\tau_i \in \Linc(F(V),F(U))$ such that $\rho_{V_i,U} \circ \tau_i = \rho_{V_i,V}$. For each $i$ we define the mapping $f_i \in \Linc(\RO(V), \RO(U))$ via 
\[ f_i(\Phi) \coleq \tau_i \circ \Phi \circ \rho_{V,U}. \]
Note that 
\begin{equation}
\rho_{V_i,U} \circ f_i(\Phi) =  \rho_{V_i,V} \circ \Phi \circ \rho_{V,U}.
\label{property-f}
\end{equation}
We set 
\[ \rho^{\cE}_{V,U}(R) \coleq \sum_{i} \eta^{i}_V \circ \rho_{V,U} \circ R \circ f_i. \]
We start by showing that $\rho^{\mathcal{E}}_{V,U}(R)$ is smooth. By \cite[Lemma 3.8]{KM} it suffices to show that
$\rho_{W,V} \circ \rho^{\mathcal{E}}_{V,U}(R) \colon \RO(V) \to F(W)$ is smooth for all $W \Subset V$. Since 
\begin{equation}
\rho_{W,V} \circ \rho^{\mathcal{E}}_{V,U}(R) = \sum_{i \in J} \eta^{i}_W \circ \rho_{W,U} \circ R \circ f_i
\label{representation-on-compacts}
\end{equation}
for some finite index set $J$, this follows from the fact that $\eta^{i}_W$, $\rho_{W,U}$, and  $f_i$ are continuous linear mappings. Next, we show that $\rho^{\mathcal{E}}_{V,U}(R)$ is local. It suffices to show that for all $W \Subset V$,
\[ \rho_{W, V} \circ \Phi_1 = \rho_{W,V} \circ \Phi_2, \qquad \Phi_1, \Phi_2 \in \RO(V), \]
implies
\[ \rho_{W,V} ( \rho^{\mathcal{E}}_{V,U}(R)(\Phi_1)) = \rho_{W, V} ( \rho^{\mathcal{E}}_{V,U}(R)(\Phi_2)). \]
The mapping $\rho_{W,V} \circ \rho^{\mathcal{E}}_{V,U}(R)$ can be represented as \eqref{representation-on-compacts} for some finite index set $J$. Since $\supp \eta^{i} \subset V_i$ it suffices to show that
\[ \rho_{W\cap V_i,U}(R(f_i(\Phi_1))) =   \rho_{W\cap V_i,U}(R(f_i(\Phi_2))) \]
for all $i \in J$. By locality of $R$ this follows from \eqref{property-f} and our assumption. The linearity of the mapping $\rho^{\mathcal{E}}_{V,U}$ is clear. We now show \ref{5.13.1}. Given $W \Subset V$, the mapping $\rho_{W,V} \circ \rho^{\mathcal{E}}_{V,U}(R)$ can be represented as \eqref{representation-on-compacts} for some finite index set $J$. 
Since
\[ \rho_{W,U}(R(\Phi')) = \sum_{i \in J} \eta^{i}_W (\rho_{W,U}( R(\Phi'))) \] 
and $\supp \eta^{i} \subset V_i$ it is enough to show that 
\[ \rho_{W \cap V_i,U}(R(f_i(\Phi))) =   \rho_{W \cap V_i,U}( R(\Phi')) \]
for all $i \in J$.
Again, by locality of $R$ this follows from \eqref{property-f} and our assumption. 
The mapping $\rho^{\mathcal{E}}_{V,U}$ is unique because for any $W \Subset V$ and any $\Phi \in \RO(V)$ one can find $\Phi' \in \RO(U)$ such that $\rho_{W, V} \circ \Phi \circ \rho_{V,U} =\rho_{W, U} \circ \Phi'$;  this follows from the fact that $F$ is fine. We continue with showing \ref{5.13.2}. We use induction on $l$. The case $l = 0$ has been treated in \ref{5.13.1}. Now suppose that the statement holds for $l-1$ and let us show it for $l$.
\begin{align*}
&\rho_{W, V}((\ud^l(\rho^{\mathcal{E}}_{V,U}(R)))(\Phi)(\Psi_1, \ldots, \Psi_l)) \\
&= \rho_{W, V} \left( \left.\frac{\ud}{\ud t}\right|_{t=0} (\ud^{l-1}(\rho^{\mathcal{E}}_{V,U}(R)))(\Phi + t \Psi_1)(\Psi_2, \ldots, \Psi_l) \right) \\
&=  \left.\frac{\ud}{\ud t}\right|_{t=0} \rho_{W, V}((\ud^{l-1}(\rho^{\mathcal{E}}_{V,U}(R)))(\Phi + t \Psi_1)(\Psi_2, \ldots, \Psi_l)) \\
&= \left.\frac{\ud}{\ud t}\right|_{t=0} \rho_{W, V}((\ud^{l-1}(\rho^{\mathcal{E}}_{V,U}(R)))(\Phi' + t \Psi'_1)(\Psi'_2, \ldots, \Psi'_l)) \\ 
&= \rho_{W, V} \left( \left.\frac{\ud}{\ud t}\right|_{t=0} (\ud^{l-1}(\rho^{\mathcal{E}}_{V,U}(R)))(\Phi' + t \Psi'_1)(\Psi'_2, \ldots, \Psi'_l) \right) \\
&=\rho_{W, V}((\ud^l(\rho^{\mathcal{E}}_{V,U}(R)))(\Phi')(\Psi'_1, \ldots, \Psi'_l)). \\
\end{align*}
Finally, we prove \ref{5.13.3}. Let $R \in \cEloc(U)$ be arbitrary. It suffices to show that for all $\Phi \in \RO(W)$ and all $W_0 \Subset W$ it hold that
\[ \rho_{W_0,W} (\rho^{\mathcal{E}}_{W,V}(\rho^{\mathcal{E}}_{V,U}(R))(\Phi)) = \rho_{W_0,W} ( \rho^{\mathcal{E}}_{W,U}(R)(\Phi)). \]
Choose $\Phi' \in \RO(V)$ such that
\[ \rho_{W_0, W} \circ \Phi \circ \rho_{W,V} =\rho_{W_0, V} \circ \Phi' \]
and  $\Phi'' \in \RO(U)$ such that
\[ \rho_{W_0, V} \circ \Phi' \circ \rho_{V,U} =\rho_{W_0, U} \circ \Phi''. \]
Hence also 
\[ \rho_{W_0, W} \circ \Phi \circ \rho_{W,U} =\rho_{W_0, U} \circ \Phi''. \]
Therefore \ref{5.13.1} implies that
\begin{align*}
\rho_{W_0,W} (\rho^{\mathcal{E}}_{W,V}(\rho^{\mathcal{E}}_{V,U}(R))(\Phi)) &= \rho_{W_0,V} (\rho^{\mathcal{E}}_{V,U}(R)(\Phi')) \\ 
&=  \rho_{W_0,U} (R(\Phi'')) \\ 
& =  \rho_{W_0,W} (\rho^{\mathcal{E}}_{W,U}(R)(\Phi)). \qedhere
\end{align*}
\end{proof}
We now discuss the extension of sheaf morphisms to $\mathcal{E}$.

\begin{lemma}\label{extension-mapping-local}
Let $T\colon F \times \cdots \times F \to F$ be a multilinear sheaf morphism. For each open subset $U \subseteq X$ consider the mapping $\widetilde{T}_U \colon \mathcal{E}(U) \times \cdots \times \mathcal{E}(U)\to \mathcal{E}(U)$
given by $\widetilde{T}_U := \widetilde{T_U}$ as in \eqref{definition1}.
Then, $\widetilde{T}$ preserves locality, i.e., $\widetilde{T}_U(\cEloc(U), \ldots, \cEloc(U)) \subseteq \cEloc(U)$
and 
\[ \rho^{\mathcal{E}}_{V,U}(\widetilde{T}_U(R_1, \ldots, R_n)) =  \widetilde{T}_V ( \rho^{\mathcal{E}}_{V,U}(R_1), \ldots, \rho^{\mathcal{E}}_{V,U}(R_n)) \]
for all  open subsets $U,V$ of $X$ with $V \subseteq U$.
\end{lemma}

\begin{proof}
The mappings $T_U$ are well-defined by \Cref{extension-bs}. Moreover, the fact that the $T_U$ preserve locality is clear from their definition. In order to show the last property it suffices to show that 
\[ \rho_{W,V}(\rho^{\mathcal{E}}_{V,U}(T_U(R_1, \ldots, R_n))(\Phi)) =  \rho_{W,V}(T_V ( \rho^{\mathcal{E}}_{V,U}(R_1), \ldots, \rho^{\mathcal{E}}_{V,U}(R_n))(\Phi)) \]
for all $\Phi \in \RO(V)$ and all $W \Subset V$. Choose $\Phi' \in  \RO(U)$ such that
\[ \rho_{W, V} \circ \Phi \circ \rho_{V,U} =\rho_{W, U} \circ \Phi'. \]
\Cref{restriction-basic-space} \ref{5.13.1} implies that
\begin{align*}
\rho_{W,V}(\rho^{\mathcal{E}}_{V,U}(T_U(R_1, \ldots, R_n))(\Phi)) &= \rho_{W,U}(T_U(R_1, \ldots, R_n)(\Phi')) \\
& = \rho_{W,U}(T_U(R_1(\Phi'), \ldots, R_n(\Phi'))) \\
&= T_W(\rho_{W,U}(R_1(\Phi')), \ldots, \rho_{W,U}(R_n(\Phi'))) \\
&= T_W(\rho_{W,V}(\rho^{\mathcal{E}}_{V,U}(R_1)(\Phi)), \ldots, \rho_{W,V}(\rho^{\mathcal{E}}_{V,U}(R_n)(\Phi))) \\
& = \rho_{W,V}(T_V(\rho^{\mathcal{E}}_{V,U}(R_1)(\Phi), \ldots, \rho^{\mathcal{E}}_{V,U}(R_n)(\Phi))) \\
& = \rho_{W,V}(T_V ( \rho^{\mathcal{E}}_{V,U}(R_1), \ldots, \rho^{\mathcal{E}}_{V,U}(R_n))(\Phi)).\qedhere
\end{align*}\end{proof}

\begin{lemma}Let $(E_1, F_1)$, $(E_2, F_2)$ be test pairs of sheaves. Suppose we are given a sheaf isomorphism $\mu \colon E_1 \to E_2$ such that its restriction to $F_1$ is a sheaf isomorphism $\mu \colon F_1 \to F_2$. For each open subset $U \subseteq X$ consider the mapping $(\mu_*)_U \colon \cE(E_1(U), F_1(U)) \to \cE(E_2(U), F_2(U))$ given by $(\mu_*)_U := (\mu_U)_*$ as in \eqref{definition2}.
Then, $\mu_*$ preserves locality, i.e., $(\mu_*)_U ( \cEloc (U)) \subseteq \cEloc(U)$, and $\rho^\cE_{V, U} ( ( \mu_*)_U R ) = ( \mu_*)_V ( \rho^\cE_{V,U}(R))$.
\end{lemma}

\begin{proof}
Suppose we are given open subsets $V,U \subseteq X$ with $V \subseteq U$, $R \in \cEloc(E_1(U), F_1(U))$ and $\Phi_1, \Phi_2 \in \RO(U)$ with $\rho_{V,U} \circ \Phi_1 = \rho_{V, U} \circ \Phi_2$. We first need to show that
\[ \rho_{V,U} ( ( (\mu_*)_U R)(\Phi_1)) = \rho_{V,U} (((\mu_*)_U R)(\Phi_2)). \]
For this we notice that
\begin{gather*}
 \rho_{V,U} ( \mu_U ( R ( \mu_U^{-1} \circ \Phi_1 \circ \mu_U ))) = \mu_V ( \rho_{V,U}(R(\mu_U^{-1} \circ \Phi_1 \circ \mu_U))) \\
= \mu_V ( \rho_{V,U} ( R(\mu_U^{-1} \circ \Phi_2 \circ \mu_U))) = \rho_{V,U} ( \mu_U ( R ( \mu_U^{-1} \circ \Phi_2 \circ \mu_U ))) 
\end{gather*}
because $R$ is local and
\begin{gather*}
\rho_{V,U} \circ \mu_U^{-1} \circ \Phi_1 \circ \mu_U = \mu_V^{-1} \circ \rho_{V,U} \circ \Phi_1 \circ \mu_U\\
= \mu_V^{-1} \circ \rho_{V,U} \circ \Phi_2 \circ \mu_U = \rho_{V,U} \circ \mu_U^{-1} \circ \Phi_2 \circ \mu_U.
\end{gather*}
For the second statement it suffices to show that
\[ \rho_{W,V} ( \rho^\cE_{V,U} ( (\mu_*)_U R)(\Phi)) = \rho_{W,V} ( (\mu_*)_V ( \rho^\cE_{V,U} ( R))(\Phi)) \]
for all $W \Subset V$ and all $\Phi \in \RO(V)$. Choose $\Phi' \in \RO(U)$ such that $\rho_{W,V} \circ \Phi \circ \rho_{V,U} = \rho_{W,U} \circ \Phi'$. Then,
\begin{align*}
 \rho_{W,V} ( \rho^\cE_{V,U} ( (\mu_*)_U R)(\Phi)) &= \rho_{W,U} ( ( (\mu_*)_U R)(\Phi')) \\
&= \rho_{W,U} ( \mu_U ( R ( \mu_U^{-1} \circ \Phi' \circ \mu_U))) \\
&= \mu_W ( \rho_{W,U} ( R ( \mu_U^{-1} \circ \Phi' \circ \mu_U ))) \\
&= \mu_W ( \rho_{W,V} ( (\rho^\cE_{V,U} R)(\mu_V^{-1} \circ \Phi \circ \mu_V ))) \\
&= \rho_{W,V} ( (\mu_*)_V ( \rho^\cE_{V,U}(R)) (\Phi) )
\end{align*}
where we used that
\[ \rho_{W,V} \circ (\mu_V^{-1} \circ \Phi \circ \mu_V ) \circ \rho_{V,U} = \rho_{W,U} \circ (\mu_U^{-1} \circ \Phi' \circ \mu_U). \qedhere \]
\end{proof}

\begin{lemma}Let $T \colon E \to E$ be a sheaf morphism such that $T|_F \colon F \to F$ is a sheaf morphism. For any open subset $U \subseteq X$ consider the mapping
 \[ \widehat T_U \colon \cE(U) \to \cE(U) \]
given by $\widehat T_U \coleq \widehat{ T_U}$ as in \eqref{definition3}. Then, $\widehat T$ preserves locality, i.e., $\widehat T_U ( \cEloc(U)) \subseteq \cEloc(U)$ and $\rho^\cE_{V,U} ( \widehat T_U(R)) = \widehat T_V ( \rho^\cE_{V,U}(R))$.
\end{lemma}
\begin{proof}Suppose we are given open sets $U,V$ with $V \subseteq U$, $R \in \cEloc(U)$ and $\Phi_1, \Phi_2 \in \RO(U)$ with $\rho_{V,U} \circ \Phi_1 = \rho_{V,U} \circ \Phi_2$. We see that
\begin{align*}
\rho_{V,U} ( ( \widehat T_U R)(\Phi_1)) &= \rho_{V,U} ( T_U(R(\Phi_1)) - \ud R(\Phi_1)(T_U \circ \Phi_1 - \Phi_1 \circ T_U)) \\
&= T_V ( \rho_{V,U} ( R ( \Phi_1))) - \rho_{V,U} ( \ud R (\Phi_1)(T_U \circ \Phi_1 - \Phi_1 \circ T_U)) \\
&= T_V ( \rho_{V,U}(R(\Phi_2))) - \rho_{V,U} ( \ud R(\Phi_2)(T_U \circ \Phi_2 - \Phi_2 \circ T_U)) \\
&= \rho_{V,U} ( (\widehat T_U R)(\Phi_2)) 
 \end{align*}
because
\begin{gather*}
 \rho_{V,U} \circ (T_U \circ \Phi_1 - \Phi_1 \circ T_U) = T_V \circ \rho_{V,U} \circ \Phi_1 - \rho_{V,U} \circ \Phi_1 \circ T_U \\
= T_V \circ \rho_{V,U} \circ \Phi_2 - \rho_{V,U} \circ \Phi_2 \circ T_U = \rho_{V,U} \circ ( T_U \circ \Phi_2 - \Phi_2 \circ T_U).
\end{gather*}
For the second statement, let $W \Subset V$ and $\Phi \in \RO(V)$. Choose $\Phi' \in \RO(U)$ such that $\rho_{W,V} \circ \Phi \circ \rho_{V,U} = \rho_{W,U} \circ \Phi'$. Then,
\begin{align*}
 \rho_{W,V} ( \rho^\cE_{V,U} ( \widehat T_U (R)) (\Phi )) & = \rho_{W,U} ( \widehat T_U(R)(\Phi')) \\
& = \rho_{W,U} ( T_U(R(\Phi')) - \ud R(\Phi')(T_U \circ \Phi' - \Phi' \circ T_U)) \\
& = T_W ( \rho_{W,U} ( R(\Phi'))) - \rho_{W,V} ( \ud ( \rho^\cE_{V,U} R)(\Phi) (T_V \circ \Phi - \Phi \circ T_V)) \\
& = \rho_{W,V} ( T_V ( (\rho^\cE_{V,U} R)(\Phi)) - \ud ( \rho^\cE_{V,U} R)(\Phi)(T_V \circ \Phi - \Phi \circ T_V )) \\
& = \rho_{W,V} ( \widehat T_V ( \rho^\cE_{V,U} R)(\Phi)).\qedhere
\end{align*}
\end{proof}

We now make the quotient construction (see \Cref{quotient-global}).
\begin{definition}
Let $\cS$ be an admissible pair of scales. For any open subset $U \subseteq X$ we define the space of moderate elements of $\cEloc(U)$ (with respect to $\cS$) as
\[ \mathcal{E}_{\cM, \operatorname{loc}}(U) = \mathcal{E}_{\cM, \operatorname{loc}}(U, \cS) \coleq \mathcal{E}_{\cM}(E(U),F(U), \TOloc(U), \TOzloc(U), \cS) \cap \cEloc(U), \]
and the space of negligible elements (with respect to $\cS$) as
\[ \mathcal{E}_{\cN, \operatorname{loc}}(U) = \mathcal{E}_{\cN, \operatorname{loc}}(U, \cS) \coleq \mathcal{E}_{\cN}(E(U),F(U), \TOloc(U), \TOzloc(U), \cS) \cap \cEloc(U). \]
We set $\cGloc(U)  = \cGloc(U, \cS) \coleq \mathcal{E}_{\cM, \operatorname{loc}}(U) \slash \mathcal{E}_{\cN, \operatorname{loc}}(U)$.
\end{definition}
\begin{lemma}\label{restriction-moderate}
Let $U \subseteq X$ be open and let $(U_i)_i$ be an open covering. Let $R \in \cEloc(U)$. Then, $R$ is moderate (negligible, respectively) if and only if
$\rho^{\mathcal{E}}_{U_i,U}(R)$ is moderate (negligible, respectively) for all $i$.
\end{lemma}
\begin{proof}
Let $R \in \cEloc(U)$ be moderate or negligible. The moderateness or negligibility of $\rho^{\mathcal{E}}_{U_i,U}(R)$ is determined by
\[ p(\rho_{W,U_i}((\ud^l(\rho^{\mathcal{E}}_{U_i,U}(R)))(\Phi_\e)(\Psi_{1,\e}, \ldots, \Psi_{l,\e}))) \]
for $\e$ small enough, where $l \in \N$, $\Phi_\e \in \TOloc(U_i)$, $\Psi_{j,\e} \in \TOzloc(U_i)$ for $j = 1, \ldots, l$, $W \Subset U_i$, and $p \in \csn(F(W))$ are arbitrary. By \Cref{soft-RO} there are $(\Phi'_\e)  \in \TOloc(U)$ and $\Psi'_{j,\e} \in \TOzloc(U)$ such that
\[ \rho_{W, U_i} \circ \Phi_\e \circ \rho_{U_i,U} =\rho_{W, U} \circ \Phi'_\e,  \qquad \rho_{W, U_i} \circ \Psi_{j,\e} \circ \rho_{U_i,U} =\rho_{W, U} \circ \Psi'_{j,\e} \]
for all $j = 1, \ldots, l$ and $\e$ small enough. Hence \Cref{restriction-basic-space} \ref{5.13.2} implies that 
\[ \rho_{W,U_i}((\ud^l(\rho^{\mathcal{E}}_{U_i,U}(R)))(\Phi_\e)(\Psi_{1,\e}, \ldots, \Psi_{l,\e})) = \rho_{W,U}((\ud^l(R))(\Phi'_\e)(\Psi'_{1,\e}, \ldots, \Psi'_{l,\e})) \]
for $\e$ small enough. The moderateness or negligibility of $\rho^{\mathcal{E}}_{U_i,U}(R)$ therefore follows from the corresponding property of $R$ and the continuity of $ \rho_{W,U}$. Conversely, suppose that $\rho^{\mathcal{E}}_{U_i,U}(R)$ is moderate or negligible for all $i$. The moderateness of $R$ is determined by 
\[ p(\rho_{W,U}((\ud^lR)(\Phi_\e)(\Psi_{1,\e}, \ldots, \Psi_{l,\e}))) \]
for $\e$ small enough, where $l \in \N$, $\Phi_\e \in \TOloc(U)$, $\Psi_{j,\e} \in \TOzloc(U)$ for $j = 1, \ldots, l$, $W \Subset U_i$ (for some $i$), and $p \in \csn(F(W))$ are arbitrary. \Cref{restriction} \ref{5.4.3} and \Cref{restriction-basic-space} \ref{5.13.2} imply that
\begin{align*}
&\rho_{W,U}((\ud^lR)(\Phi_\e)(\Psi_{1,\e}, \ldots, \Psi_{l,\e})) \\
&= \rho_{W,U_i}((\ud^l(\rho^{\mathcal{E}}_{U_i,U}(R)))(\rho^{\RO}_{U_i,U}(\Phi_\e))(\rho^{\RO}_{U_i,U}(\Psi_{1,\e}), \ldots, \rho^{\RO}_{U_i,U}(\Psi_{l,\e}))) 
\end{align*}
for $\e$ small enough. The moderateness or negligibility of $R$ therefore follows from the corresponding property of $\rho^{\mathcal{E}}_{U_i,U}(R)$ and the continuity of $ \rho_{W,U_i}$. 
\end{proof}
\Cref{restriction-basic-space} and \Cref{restriction-moderate} imply that the mappings
\[ \rho^{\cG}_{V,U}([R]) \coleq [\rho^{\mathcal{E}}_{V,U}(R)] \]
define a presheaf structure on $U \to \cGloc(U)$. We now show that it is in fact a sheaf.
\begin{proposition}
$\cGloc$ is a sheaf of vector spaces.
\end{proposition}
\begin{proof}
\ref{S1} Immediate consequence of \Cref{restriction-moderate}.

\ref{S2} Let $U \subseteq X$ be open and let $(U_i)_i$ be an open covering of $U$. Since $X$ is locally compact we may assume without loss of generality that $U_i \Subset U$ for all $i$. Suppose that $[R_i] \in \cGloc(U_i)$ are given such that $\rho^{\cG}_{U_i \cap U_j,U_i}([R_i]) = \rho^{\cG}_{U_i \cap U_j,U_j}([R_j])$ for all $i,j$. Let $(\eta^{i})_{i} \subset \Hom(F|_U,F|_U)$
be a partition of unity subordinate to $(U_i)_i$. Choose $\tau_{i}  \in \Linc(F(U_i), F(U))$ such that $\rho_{V_i,U} \circ \tau_i = \rho_{V_i, U_i}$ for some  $V_i \Subset U_i$ with $\supp \eta^i \subset V_i$. We define
\[ R \coleq \sum_{i} \eta^{i}_U \circ \tau_i \circ R_i \circ \rho^{\RO}_{U_i,U}. \]
We start with showing that $R \in C^{\infty}(\RO(U), F(U))$. By \cite[Lemma 3.8]{KM} it suffices to show that
$\rho_{W,U} \circ R \colon \RO(U) \to F(W)$ is smooth for all $W \Subset U$. Since 
\begin{equation}
\rho_{W,U} \circ R = \sum_{i \in J} \eta^{i}_W \circ \rho_{W,U} \circ \tau_i \circ R_i \circ \rho^{\RO}_{U_i,U} 
\label{representation-on-compacts-2}
\end{equation}
for some finite index set $J$, this follows from the fact that the linear mappings $\eta^{i}_W$, $\rho_{W,U}$, $\tau_i$, and $\rho^{\RO}_{U_i,U}$ are continuous (see  \Cref{restriction}). Next, we show that $R$ is local. We need to show that for all $W \Subset U$ the equality
\[ \rho_{W, U} \circ \Phi_1 = \rho_{W,U} \circ \Phi_2, \qquad \Phi_1, \Phi_2 \in \RO(U), \]
implies
\[ \rho_{W,U} (R(\Phi_1)) = \rho_{W, U} (R(\Phi_2)). \]
The mapping $\rho_{W,U} \circ R$ can be represented as \eqref{representation-on-compacts-2} for some finite index set $J$. Since $\supp \eta^{i} \subset V_i$ and $\rho_{V_i,U} \circ \tau_i = \rho_{V_i, U_i}$, it suffices to show that
\[ \rho_{W\cap V_i,U_i}(R_i( \rho^{\RO}_{U_i,U}(\Phi_1))) = \rho_{W\cap V_i,U_i}(R_i( \rho^{\RO}_{U_i,U}(\Phi_2))) \]
for all $i \in  J$. By locality of $R_i$ this follows from \Cref{restriction} \ref{5.4.4} and our assumption. We continue with showing that $R$ is moderate. The moderateness of $R$ is determined by the values of the sequence
\[ p(\rho_{W,U}((\ud^lR)(\Phi_\e)(\Psi_{1,\e}, \ldots, \Psi_{l,\e}))) \]
for $\e$ small enough, where $l \in \N$, $\Phi_\e \in \TOloc(U)$, $\Psi_{j,\e} \in \TOzloc(U)$ for $j = 1, \ldots, l$, $W \Subset U$, and $p \in \csn(F(W))$ are arbitrary. Since 
\begin{align*}
&\rho_{W,U}((\ud^lR)(\Phi_\e)(\Psi_{1,\e}, \ldots, \Psi_{l,\e})) \\
&=( \ud^l(\rho_{W,U} \circ R))(\Phi_\e)(\Psi_{1,\e}, \ldots, \Psi_{l,\e}) \\
&= \left( \ud^l \left( \sum_{i \in J} \eta^{i}_W \circ \rho_{W,U} \circ \tau_i \circ R_i \circ \rho^{\RO}_{U_i,U} \right) \right) (\Phi_\e)(\Psi_{1,\e}, \ldots, \Psi_{l,\e}) \\
& = \sum_{i \in J} (\eta^{i}_W \circ \rho_{W,U} \circ \tau_i) ((\ud^lR_i)(\rho^{\RO}_{U_i,U}(\Phi_\e)) (\rho^{\RO}_{U_i,U}(\Psi_{1,\e}), \ldots, \rho^{\RO}_{U_i,U}(\Psi_{l,\e})))
\end{align*}
for some finite index set $J$, the moderateness of $R$ follows from the continuity of the mapping $\eta^{i}_W \circ \rho_{W,U} \circ \tau_i$ and the moderateness of the $R_i$. Finally, we show that 
$\rho^{\cG}_{U_i,U}([R]) = [R_i]$ for all $i$. We need to show that $\rho^{\mathcal{E}}_{U_i,U}(R) - R_i$ is negligible. The negligibility is determined by 
\[ p(\rho_{W,U_i}((\ud^l(\rho^{\mathcal{E}}_{U_i,U}(R) - R_i))(\Phi_\e)(\Psi_{1,\e}, \ldots, \Psi_{l,\e}))) \]
for $\e$ small enough, where $l \in \N$, $\Phi_\e \in \TOloc(U_i)$, $\Psi_{j,\e} \in \TOzloc(U_i)$ for $j = 1, \ldots, l$, $W \Subset U_i$, and $p \in \csn(F(W))$ are arbitrary. By \Cref{soft-RO}  there are  $\Phi'_\e \in \TOloc(U)$, $\Psi'_{j,\e} \in \TOzloc(U)$ for $j = 1, \ldots, l$ such that
\[ \rho_{W, U_i} \circ \Phi_\e \circ \rho_{U_i,U} =\rho_{W, U} \circ \Phi'_\e,  \qquad \rho_{W, U_i} \circ \Psi_{j,\e} \circ \rho_{U_i,U} =\rho_{W, U} \circ \Psi'_{j,\e} \]
for all $j= 1, \ldots, l$ and $\e$ small enough. Hence \Cref{restriction} \ref{5.4.2} yields that
\begin{align*}
&\rho_{W,U_i}((\ud^l(\rho^{\mathcal{E}}_{U_i,U}(R)))(\Phi_\e)(\Psi_{1,\e}, \ldots, \Psi_{l,\e})) \\
&= \rho_{W,U}((\ud^lR)(\Phi'_\e)(\Psi'_{1,\e}, \ldots, \Psi'_{l,\e}))\\
&= (\ud^l(\rho_{W,U} \circ R))(\Phi'_\e)(\Psi'_{1,\e}, \ldots, \Psi'_{l,\e}) \\
&= \left(\ud^l \left( \sum_{j \in J} \eta^{j}_W \circ \rho_{W,U} \circ \tau_j \circ R_j \circ \rho^{\RO}_{U_j,U} \right)\right)(\Phi'_\e)(\Psi'_{1,\e}, \ldots, \Psi'_{l,\e}) \\
&= \sum_{j \in J} \eta^{j}_W ( \rho_{W,U} ( \tau_j ((\ud^lR_j)(\rho^{\RO}_{U_j,U}(\Phi'_\e)) (\rho^{\RO}_{U_j,U}(\Psi'_{1,\e}), \ldots, \rho^{\RO}_{U_j,U}(\Psi'_{l,\e})))))
\end{align*}
for $\e$ small enough. On the other hand, \Cref{restriction} \ref{5.4.2} and the fact that $(\Phi_\e)_\e$ is localizing imply that
\[ \rho_{W,U_i} \circ \Phi_\e = \rho_{W,U_i} \circ \rho^{\RO}_{U_i,U}(\Phi'_\e), \qquad  \rho_{W,U_i} \circ \Psi_{j,\e} = \rho_{W,U_i} \circ \rho^{\RO}_{U_i,U}(\Psi'_{j,\e}) \]
for all $j = 1, \ldots, l$ and  $\e$ small enough.  By \Cref{local-derivatives} we obtain that
\begin{align*}
& \rho_{W,U_i}((\ud^lR_i)(\Phi_\e)(\Psi_{1,\e}, \ldots, \Psi_{l,\e}))  \\
&= \rho_{W,U_i}((\ud^lR_i)(\rho^{\RO}_{U_i,U}(\Phi'_{\e}))(\rho^{\RO}_{U_i,U}(\Psi'_{1,\e}), \ldots, \rho^{\RO}_{U_i,U}(\Psi'_{l,\e})))  \\
&=\sum_{j \in J} \eta^{j}_W (\rho_{W,U_i}((\ud^lR_i)(\rho^{\RO}_{U_i,U}(\Phi'_{\e}))(\rho^{\RO}_{U_i,U}(\Psi'_{1,\e}), \ldots, \rho^{\RO}_{U_i,U}(\Psi'_{l,\e}))))
\end{align*}
for $\e$ small enough. Since $\supp \eta^{j} \subset V_j$ and $\rho_{V_j,U} \circ \tau_j= \rho_{V_j, U_j}$, it suffices to estimate
\begin{align*}
 &\rho_{W \cap V_j,U_j}((\ud^lR_j)(\rho^{\RO}_{U_j,U}(\Phi'_\e)) (\rho^{\RO}_{U_j,U}(\Psi'_{1,\e}), \ldots, \rho^{\RO}_{U_j,U}(\Psi'_{l,\e}))) \\ 
 & - \rho_{W \cap V_j ,U_i}((\ud^lR_i)(\rho^{\RO}_{U_i,U}(\Phi'_{\e}))(\rho^{\RO}_{U_i,U}(\Psi'_{1,\e}), \ldots, \rho^{\RO}_{U_i,U}(\Psi'_{l,\e}))).
\end{align*}
for all $j \in J$. By \Cref{restriction-basic-space} \ref{5.13.2} we have that
\begin{align*}
 &\rho_{W \cap V_j,U_j}((\ud^lR_j)(\rho^{\RO}_{U_j,U}(\Phi'_\e)) (\rho^{\RO}_{U_j,U}(\Psi'_{1,\e}), \ldots, \rho^{\RO}_{U_j,U}(\Psi'_{l,\e}))) \\ 
 &= \rho_{W \cap V_j,U_i \cap U_j}((\ud^l(\rho^{\mathcal{E}}_{U_i \cap U_j, U_j}(R_j)))(\rho^{\RO}_{U_i \cap U_j,U}(\Phi'_\e)) (\rho^{\RO}_{U_i \cap U_j,U}(\Psi'_{1,\e}), \ldots, \rho^{\RO}_{U_i \cap U_j,U}(\Psi'_{l,\e}))) 
\end{align*}
and
\begin{align*}
&\rho_{W \cap V_j ,U_i}((\ud^lR_i)(\rho^{\RO}_{U_i,U}(\Phi'_{\e}))(\rho^{\RO}_{U_i,U}(\Psi'_{1,\e}), \ldots, \rho^{\RO}_{U_i,U}(\Psi'_{l,\e}))) \\
&= \rho_{W \cap V_j ,U_i \cap U_j}((\ud^l(\rho^{\mathcal{E}}_{U_i \cap U_j,U_i}(R_i)))(\rho^{\RO}_{U_i \cap U_j,U}(\Phi'_{\e}))(\rho^{\RO}_{U_i \cap U_j,U}(\Psi'_{1,\e}), \ldots, \rho^{\RO}_{U_i \cap U_j,U}(\Psi'_{l,\e}))).  
\end{align*}
The negligibility now follows from the assumption.
\end{proof}
Next, we discuss the embedding of $E$ into $\cGloc$. For $U \subseteq X$ open consider the canonical embeddings (see \Cref{basic-space}) 
\[ \iota_U\colon E(U) \to \mathcal{E}(U), \qquad \sigma_U\colon F(U) \to \mathcal{E}(U). \]
Clearly, $\iota_U(E(U)) \subseteq \cEloc(U)$ and $\sigma_U(F(U)) \subseteq \cEloc(U)$.  Hence \Cref{properties-embedding} implies that the mappings
\begin{gather*}
\iota_U \colon E(U) \to \cGloc(U),\, \iota(u) \coleq [\iota(u)] \\
\sigma \colon F(U) \to \cGloc(U),\,\sigma(\varphi) \coleq [\sigma(\varphi)]
\end{gather*}
are linear embeddings such that $\iota_{U|F(U)} = \sigma_U$.
\begin{proposition}
The embeddings $\iota\colon E \to \cGloc$ and  $\sigma\colon F \to \cGloc$ are sheaf morphisms, and $\iota|_F = \sigma$.
\end{proposition}
\begin{proof}
We already noticed that  $\iota_{U|F(U)} = \sigma_U$ for all $U \subseteq X$ open. Since $F$ is a subsheaf of $E$ it therefore suffices to show that $\iota$ is a sheaf morphism. Let $U,V$ be open subsets of $X$ such that $V \subseteq U$. We need to show that for all $u \in E(U)$ it holds that
\[ \rho^{\mathcal{E}}_{V,U}(\iota_U(u)) -\iota_V(\rho_{V,U}(u)) \]
is negligible. It suffices to show that for all $W \Subset V$ and all $(\Phi_\e)_\e \in \ROloc(V)$ it holds that
\[ \rho_{W,V}(\rho^{\mathcal{E}}_{V,U}(\iota_U(u))(\Phi_\e)) = \rho_{W,V}(\iota_V(\rho_{V,U}(u))(\Phi_\e)) \]
for $\e$ small enough. By \Cref{soft-RO} there is $(\Phi'_\e)_\e \in \ROloc(U)$ such that
\[ \rho_{W,V} \circ \Phi_\e \circ \rho_{V,U} = \rho_{W,U} \circ \Phi'_\e. \]
Hence \Cref{restriction-basic-space} \ref{5.13.1} yields that
\begin{align*}
\rho_{W,V}(\rho^{\mathcal{E}}_{V,U}(\iota_U(u))(\Phi_\e)) &= \rho_{W,U}(\iota_U(u)(\Phi'_\e)) \\
&= \rho_{W,U}(\Phi'_\e(u)) = \rho_{W,V}(\Phi_\e(\rho_{V,U}(u)))\\
&=  \rho_{W,V}(\iota_V(\rho_{V,U}(u))(\Phi_\e))
\end{align*}
for $\e$ small enough.
\end{proof}
We end this section by showing how one can extend sheaf morphisms to $\cGloc$. 
\begin{lemma}
Let $T \colon F \times \cdots \times F \to F$ be a multilinear sheaf morphism.
The mappings
$\widehat{T}_U \colon \cGloc(U) \times \cdots \times  \cGloc(U) \to \cGloc(U)$ given by 
\[ \widehat{T}_U([R_1], \ldots [R_n]) \coleq [\widehat{T}_U(R_1, \ldots, R_n)] \]
are well-defined multilinear mappings such that 
\[ \widehat{T}_U(\sigma_U(\varphi_1), \ldots, \sigma_U(\varphi_n)) = \sigma_U(T_U(\varphi_1, \ldots, \varphi_n)). \]
Moreover, $T$ is a multilinear sheaf morphism.
\end{lemma}

\begin{lemma}\label{quotdiffeo}
Let $(E_1, F_1)$, $(E_2, F_2)$ be test pairs of sheaves. Suppose we are given a sheaf isomorphism $\mu \colon E_1 \to E_2$ such that its restriction to $F_1$ is a sheaf isomorphism $\mu \colon F_1 \to F_2$. The mappings $(\mu_*)_U \colon \cGloc(E_1(U), F_1(U)) \to \cGloc(E_2(U), F_2(U))$ given by $(\mu_*)_U [R] \coleq [ (\mu_*)_U R ]$ are well-defined multilinear mappings such that $(\mu_*)_U \circ \iota_U = \iota_U \circ \mu_U$ and $(\mu_*)_U \circ \sigma_U = \sigma_U \circ \mu_U$.
\end{lemma}

\begin{lemma} Let $T \colon E \to E$ be a sheaf morphism such that $T|_F \colon F \to F$ is a sheaf morphism. Then, the mappings $\widehat T_U \colon \cGloc(U) \to \cGloc(U)$ given by $\widehat T_U [ R ] \coleq [ \widehat T_U R ]$ are well-defined such that $\widehat T_U \circ \iota_U = \iota_U \circ \widehat T_U$ and $\widehat T_U \circ \sigma_U = \sigma_U \circ \widehat T_U$.
\end{lemma}

We obtain the following two important corollaries.
\begin{corollary}
For every open set $U$ in $X$ the sheaf $\cGloc|_U$ is fine.
\end{corollary}
\begin{proof}
Let $A$ and $B$ be closed sets in $U$ such that $A \cap B = \emptyset$. Let $\tau \in \Hom(F|_U,F|_U)$ be such that $\tau_V = \id$ and $\tau_W = 0$ for some open neighbourhoods $V$ and $W$ (in $U$) of $A$ and $B$, respectively. Consider the associated sheaf morphism $\tau \in \Hom(\cGloc|_U, \cGloc|_U)$. Then,
\[ \tau_V([R]) = [\tau_V(R)] = [\tau_V \circ R] = [R] \]
for all $[R] \in \cGloc(V)$. Similarly, one can show that $\tau_W = 0$.
\end{proof}
\begin{corollary}
Suppose that $F$ is a locally convex sheaf of algebras. Then, $\cGloc$ is a sheaf of algebras and the $\sigma$-embedding is a sheaf homomorphism of algebras.
\end{corollary}

\section{Diffeomorphism invariant algebras of distributions}\label{sec_appldist}

The space of distributions on a paracompact Hausdorff manifold $M$ is defined as
\[ \cD'(M) \coleq ( \Gamma_c ( M, \Vol(M) ) )' \]
where $\Gamma_c (M, \Vol(M))$ denotes the space of compactly supported sections of the volume bundle $\Vol(M)$, endowed with its natural (LF)-topology (see \cite[Section 3.1]{GKOS}). It is well known that $\cD'$ and $C^\infty$ are locally convex sheaves on $M$, so $(\cD', C^\infty)$ is a test pair of sheaves.

Given any open subset $U \subseteq M$, the space $\TOloc(U)$ is nonempty -- we refer to \cite{bigone} for the concrete construction of localizing test objects for $(\cD', C^\infty)$, which is done by convolution with smooth mollifiers in local charts.

Fix the asymptotic scale to be the polynomial scale \eqref{polyscale}. By the previous section we obtain a fine sheaf $\cGloc$ of algebras such that $\sigma \colon C^\infty \to \cGloc$ is a sheaf homomorphism of algebras and $\iota \colon \cD' \to \cGloc$ is an injective sheaf homomorphism of vector spaces. Given any vector field $X$ on $M$, the Lie derivative $L_X$ of distributions and smooth functions satisfies the assumptions of \Cref{extension-der}, so it defines a mapping
\[ \widehat L_X \colon \cGloc(M) \to \cGloc(M) \]
which commutes with $\iota$.

Moreover, given any diffeomorphism $\mu \colon M \to N$, we apply \Cref{quotdiffeo} to the functors $E_1 = \cD'$, $E_2 = \cD' \circ \mu$, $F_1 = C^\infty$, $F_2 = C^\infty \circ \mu$, which gives an induced action
\[ \mu_* \colon
 \cGloc(M) \to \cGloc(N) \]
which commutes with $\iota$.

Hence, we have easily obtained the following result of \cite{global}:

\begin{theorem}\label{thmone}Let $M$ be a paracompact Hausdorff manifold. There is an associative commutative algebra $\cGloc(M)$ with unit containing $\cD'(M)$ injectively as a linear subspace and $C^\infty(M)$ as a subalgebra. $\cGloc(M)$ is a differential algebra, where the derivations $\widehat L_X$ extend the usual Lie derivatives from $\cD'(M)$ to $\cGloc(M)$, and $\cGloc$ is a fine sheaf of algebras over $M$.
\end{theorem}

\section{Diffeomorphism invariant algebras of ultradistributions}\label{sec_applications}
\subsection{Spaces of ultradifferentiable functions and their duals.}\label{preliminaries}
Let $(M_p)_{p \in \N}$ be a sequence of positive reals (with $M_0 = 1$). We will make use of the following conditions:
\begin{enumerate}
\item[$(M.1)$] $M_p^2 \leq M_{p-1}M_{p+1}$, $p \in \Z_+$,
\item[$(M.2)$] $M_{p+q} \leq AH^{p+q}M_pM_q$, $p,q \in \N$, for some $A,H \geq 1$,
\item[$(M.3)'$] $\displaystyle \sum_{p = 1}^\infty \frac{M_{p-1}}{M_p} < \infty$.
\end{enumerate}
We refer to \cite{Komatsu} for the meaning of these conditions. For $\alpha \in \N^d$ we write $M_\alpha = M_{|\alpha|}$.  As usual, the relation $M_p\subset N_p$ between two weight sequences means that there are $C,h>0$ such that 
$M_p\leq Ch^{p}N_{p},$ $p\in\N$. The stronger relation $M_p\prec N_p$ means that the latter inequality remains valid for every $h>0$ and a suitable $C=C_{h}>0$. The associated function of $M_p$ is defined as
\[ M(t)\coleq\sup_{p\in\N}\log\frac{t^p}{M_p},\qquad t > 0, \]
and $M(0) := 0$. We define $M$ on $\R^d$ as the radial function $M(x) = M(|x|)$, $x \in \R^d$. Under $(M.1)$, the assumption $(M.2)$ holds \cite[Prop 3.6]{Komatsu} if and only if
\[ 2M(t) \leq M(Ht) + \log A, \qquad  t > 0. \]
Unless otherwise explicitly stated, $M_p$ will \emph{always} stand for a weight sequence satisfying $(M.1)$, $(M.2)$, $(M.3)'$.

For a regular compact set $K$ in $\R^d$ and $h > 0$ we write $\mathcal{E}^{M_p,h}(K)$ for the Banach space consisting of all $\varphi \in C^\infty(K)$ such that
\begin{equation}
\| \varphi \|_{K,h} \coleq\sup_{\alpha \in \N^d} \sup_{x \in K} \frac{|\varphi^{(\alpha)}(x)|}{h^{|\alpha|}M_{|\alpha|}} < \infty. 
\label{norm}
\end{equation}
The space $\mathcal{D}^{M_p,h}_K$ consists of all  $\varphi \in C^\infty(\R^d)$ with support in $K$ that satisfy \eqref{norm}. Let $\Omega \subseteq \R^d$ be open. We define
\[ \mathcal{E}^{(M_p)}(\Omega) = \varprojlim_{K \Subset \Omega} \varprojlim_{h \to 0^+}\mathcal{E}^{M_p,h}(K), \qquad \mathcal{E}^{\{M_p\}}(\Omega) = \varprojlim_{K \Subset \Omega} \varinjlim_{h \to \infty}\mathcal{E}^{M_p,h}(K), \]
and 
\[ \mathcal{D}^{(M_p)}(\Omega) = \varinjlim_{K \Subset \Omega} \varprojlim_{h \to 0^+}\mathcal{D}^{M_p,h}_K, \qquad \mathcal{D}^{\{M_p\}}(\Omega) = \varinjlim_{K \Subset \Omega} \varinjlim_{h \to \infty}\mathcal{D}^{M_p,h}_K. \]
Elements of $\mathcal{E}^{(M_p)}(\Omega)$ and $\mathcal{E}^{\{M_p\}}(\Omega)$ are called \emph{ultradifferentiable functions of class $(M_p)$ of Beurling type on $\Omega$} and \emph{ultradifferentiable functions of class $\{M_p\}$ of Roumieu type on $\Omega$}, respectively.
These spaces are complete Montel locally convex algebras (under pointwise multiplication) \cite[Th.\ 2.6, Th.\ 5.12, Th.\ 2.8]{Komatsu}. Elements of the dual spaces $\mathcal{D}'^{(M_p)}(\Omega)$ and $\mathcal{D}'^{\{M_p\}}(\Omega)$ are called \emph{ultradistributions of class $(M_p)$ of Beurling type on $\Omega$} and \emph{ultradistributions of class $\{M_p\}$ of Roumieu type on $\Omega$}, respectively. We endow these spaces with the strong topology. $\mathcal{D}'^{(M_p)}(\Omega)$ and $\mathcal{D}'^{\{M_p\}}(\Omega)$ are complete Montel locally convex spaces \cite[Th.\ 2.6]{Komatsu} and $\Omega' \to \mathcal{D}'^{(M_p)}(\Omega')$, $\Omega' \to \mathcal{D}'^{\{M_p\}}(\Omega')$  are locally convex sheaves on $\Omega$ \cite[Th.\ 5.6]{Komatsu}.

We write $\mathcal{R}$ for the family of positive real sequences $(r_j)_{j \in \N}$ (with $r_0 =1$) which increase to infinity. This set is partially ordered and directed by the relation $r_j \preceq s_j$, which means that there is an $j_0 \in \N$ such that $r_j \leq s_j$ for all $j \geq j_0$. By \cite[Prop.\ 3.5]{Komatsu3} a function $\varphi \in C^\infty(\Omega)$ belongs to $\mathcal{E}^{\{M_p\}}(\Omega)$ if and only if 
\[ \| \varphi \|_{K, r_j} \coleq \sup_{\alpha \in \N^d}\sup_{x \in K} \frac{|\varphi^{(\alpha)}(x)|}{M_{\alpha}\prod_{j = 0}^{|\alpha|}r_j} < \infty, \]
for all $K \Subset \Omega$ and $r_j \in \mathcal{R}$. Moreover, the topology of $\mathcal{E}^{\{M_p\}}(\Omega)$ is generated by the system of seminorms $\{\| \, \|_{K, r_j} \, : \, K \Subset \Omega, r_j \in \mathcal{R}\}$.

In the sequel we shall write $\ast$ instead of $(M_p)$ or $\{M_p\}$ if we want to treat both cases simultaneously. In addition, we shall often first state assertions for the Beurling case followed in parenthesis by the corresponding statements for the Roumieu case. 

\subsection{Nonlinear extensions of spaces of ultradistributions}

We apply the general theory developed in \Cref{global-theory} and \Cref{local-theory} to construct algebras containing spaces of ultradistributions which are invariant under real-analytic diffeomorphisms. 
Let us remark that this construction is a novelty, as the previous construction in \cite{D-V-V} was given in the context of special Colombeau algebras and therefore cannot be diffeomorphism invariant.

In order to not having to develop the theory of ultradistributions on manifolds here, we restrict the considerations to the local case, i.e., to open subsets of $\R^n$, where diffeomorphism invariance can be stated easily.

By the remarks in \Cref{preliminaries} and the existence of partitions of unity of ultradifferentiable functions of class $\ast$ \cite[Prop.\ 5.2]{Komatsu} it is clear that the pair $(\mathcal{D}'^\ast, \mathcal{E}^\ast)$ is a test pair of sheaves on $\R^d$, giving rise to the corresponding presheaf $\cEloc$ of basic spaces (\Cref{basicspaces}).

We now choose appropriate asymptotic scales. Given $r_j \in \mathcal{R}$ we write $M_{r_j}$ for the associated function of the weight sequence $M_p \prod_{j=0}^pr_j$.

\begin{definition} We define
\begin{gather*}
\cA^{(M_p)} \coleq \{ e^{M(\lambda/\e)} \ : \ \lambda > 0 \},  \qquad \cI^{(M_p)} \coleq \{ e^{-M(\lambda/\e)} \ : \ \lambda > 0 \}, \\
\cA^{\{M_p\}} \coleq \{ e^{M_{r_j}(1/\e)} \ : \ r_j \in \mathcal{R} \},  \qquad \cI^{\{M_p\}} \coleq \{ e^{-M_{r_j}(1/\e)} \ : \ r_j \in \mathcal{R} \}.
\end{gather*}
Condition $(M.2)$ ensures that $\operatorname{sc}^{\ast} \coleq (\cA^{\ast}, \cI^{\ast})$ are admissible pair of scales \footnote{We do not use the notation $\cS^\ast$ for the pair of scales $(\cA^{\ast}, \cI^{\ast})$ since this is the standard notation for Gelfand-Shilov type spaces.}. For $\Omega \subseteq \R^d$ open we set
\[ \TO^\ast_{\operatorname{loc}}(\Omega) \coleq \TOloc(\Omega, \mathcal{D}'^\ast, \mathcal{E}'^\ast, \operatorname{sc}^{\ast}), \qquad \TO^{0, \ast}_{\operatorname{loc}}(\Omega) \coleq \TOzloc(\Omega, \mathcal{D}'^\ast, \mathcal{E}'^\ast, \operatorname{sc}^{\ast}). \]
\end{definition}
\begin{remark} \label{test-objects-Roumieu}
It follows from \cite[Prop.\ 4.4]{D-V-V} that $(\Phi_\e)_\e \in \Linc(\mathcal{D}'^{\{M_p\}}(\Omega), \mathcal{E}^{\{M_p\}}(\Omega))^I$ satisfies $(\TO)_1$ and $(\TO)_2$ (with respect to the scale 
$\operatorname{sc}^{\{M_p\}}$) if and only if 
\begin{enumerate}[label=(\roman*)]
\item $\forall u \in  \mathcal{D}'^{\{M_p\}}(\Omega) \, \forall K \Subset \Omega \, \forall \lambda > 0 \, \exists h > 0 : \| \Phi_\e(u)\|_{K,h}  = O(e^{M(\lambda/\e)})$, 
\item $\forall \varphi \in  \mathcal{E}^{\{M_p\}}(\Omega) \, \forall K \Subset \Omega \, \exists \lambda > 0 \, \exists  h > 0 : \| \Phi_\e(\varphi) - \varphi \|_{K,h}  = O(e^{-M(\lambda/\e)})$. 
\end{enumerate}
\end{remark}
In order to be able to apply the results from \Cref{local-theory} we must show that $\TO^\ast_{\operatorname{loc}}(\Omega)$ is nonempty for every open set $\Omega \subseteq \R^d$. For this we shall need the following lemma.
\begin{lemma} \label{M_p-N_p} 
Let $M_p$ and $N_p$ be two weight sequences satisfying $(M.1)$ such that $N_p \prec M_p$ and let $M$ and $N$ be the associated functions of $M_p$ and $N_p$, respectively.  Then, there is an increasing net $(r_\e)_{\e}$ of positive reals with $\displaystyle \lim_{\e \to 0^+}r_\e = 0$  such that for every $\lambda > 0$ there is $\e_0 > 0$ such that for all $\e < \e_0$ it holds that
\[ M(t) \leq N(r_\e t) + M(\lambda/ \e), \qquad t > 0. \]
\end{lemma}
\begin{proof}
By \cite[Lemma 3.10]{Komatsu} there is a continuous increasing function $\rho\colon (0,\infty) \to (0,\infty)$ with
\[ \lim_{t \to \infty} \frac{\rho(t)}{t} = 0 \]
such that $M(t) = N(\rho(t))$ for all $t > 0$. One can readily verify that 
\[ r_\e \coleq \sup_{t \geq 1/\sqrt{\e}}  \frac{\rho(t)}{t} \]
satisfies all requirements.
\end{proof}
By  \cite[Lemma 4.3]{Komatsu}
there is a weight sequence $N_p$ satisfying $(M.1)$ and $(M.3)'$ such that $N_p \prec M_p$. Pick $\psi \in \mathcal{D}^{(M_p)}(\R^d)$ even with $0 \leq \psi \leq 1$, $\supp \psi \subset B(0,2)$, and $\psi \equiv 1$ on $\overline{B}(0,1)$, and $\chi \in \mathcal{D}^{(N_p)}(\R^d)$ even with $\supp \chi \subset B(0,2)$ and $\chi \equiv 1$ on $\overline{B}(0,1)$. Choose $(r_\e)_\e$ according to \Cref{M_p-N_p}. We define
\[ \theta_\e(x) \coleq \frac{1}{\e^d} \mathcal{F}^{-1}(\psi)(x/\e)\chi(x/r_\e), \qquad x \in \R^d, \]
where we fix the constants in the Fourier transform as follows
\[ \mathcal{F}(\varphi)(\xi)  =\widehat{\varphi}(\xi) \coleq \int_{\R^d}\varphi(x) e^{-ix\xi}\,\dx. \]
Next, let $(K_n)_{n \in \N}$ be an exhaustion by compacts of $\Omega$ and choose $\kappa_n \in \mathcal{D}^{(M_p)}(\Omega)$ such that $\kappa_n \equiv 1$ on $K_n$. For $\e \in I$ we set $\kappa_\e = \kappa_n$ if $n \leq \e^{-1} < n+1$. Finally, we define
\[ \Phi_\e (u)\coleq (\kappa_\e u) \ast \theta_\e = \langle u(x), \kappa_\e(x)\theta_\e(\cdot - x) \rangle, \qquad u \in \mathcal{D}'^\ast(\Omega). \]
\begin{lemma} \label{RO-ultradistributions}
$(\Phi_\e)_\e \in \TO^\ast_{\operatorname{loc}}(\Omega)$. 
\end{lemma}
The proof of \Cref{RO-ultradistributions} is based on the following growth estimates of the Fourier transforms of the $\theta_\e$.
\begin{lemma}\label{growth-fourier-theta}
$\mbox{}$
\begin{enumerate}[label=(\roman*)]
\item\label{7.5.1} For all $\e \in I$ it holds that
\[ \sup_{\xi \in \R^d}|\widehat{\theta}_\e(\xi)| \leq \frac{1}{(2\pi)^d} \|\widehat{\chi}\|_{L^1(\R^d)}, \]
\item\label{7.5.2} for all $h, \lambda > 0$ there is $\e_0 > 0$ such that 
\[ \sup_{\e < \e_0} \sup_{|\xi| \geq 4/\e} |\widehat{\theta}_\e(\xi)| e^{M(\xi/h) - M(\lambda/\e)} < \infty. \]
\item\label{7.5.3} for all $\lambda > 0$ there is $\e_0 > 0$ such that 
\[ \sup_{\e < \e_0} \sup_{|\xi| \leq 2/\e} |1 - \widehat{\theta}_\e(\xi)| e^{M(\lambda/\e)} < \infty. \]
\end{enumerate}
\end{lemma}
\begin{proof}
Property \ref{7.5.1} is clear. We now show \ref{7.5.2}. Let $\e \in I$ be arbitrary. We have that 
\begin{align*}
|\widehat{\theta}_\e(\xi)| &= \frac{r^d_\e}{(2\pi)^d} \left | \int_{\R^d} \psi(\e \eta) \widehat{\chi}(r_\e(\xi- \eta)) \deta \right| \\
&\leq  \frac{r^d_\e}{(2\pi)^d}  \int_{|\eta| \leq \frac{2}{\e}} |\widehat{\chi}(r_\e(\xi- \eta))| \deta \\
&= \frac{1}{(2\pi)^d}  \int_{\left| \xi - \frac{t}{r_\e} \right| \leq \frac{2}{\e}} |\widehat{\chi}(t)| \dt.
\end{align*} 
By \cite[Lemma 3.3]{Komatsu} there is $C > 0$ such that
\[  |\widehat{\chi}(t)| \leq Ce^{-N(2Ht/h)} \leq ACe^{-2N(2t/h)}, \qquad t \in \R^d. \]
Furthermore, notice that for $\xi, t \in \R^d$ it holds that
\[ |\xi| \geq \frac{4}{\e} \mbox{ and } \left| \xi - \frac{t}{r_\e} \right| \leq \frac{2}{\e} \to |t| \geq \frac{r_\e |\xi|}{2}. \]
Hence we obtain that
\[ |\widehat{\theta}_\e(\xi)| \leq C'e^{-N(r_\e|\xi|/h)},\qquad |\xi| \geq \frac{4}{\e}, \]
where 
\[ C' = \frac{AC}{(2\pi)^d}\int_{\R^d} e^{-N(2t/h)} \dt < \infty. \]
The result now follows from \Cref{M_p-N_p}. Finally, we show \ref{7.5.3}. Let $\e \in I$ be arbitrary. We have that 
\begin{align*}
|1 - \widehat{\theta}_\e(\xi)| &= \left | 1 - \frac{r^d_\e}{(2\pi)^d}  \int_{\R^d} \psi(\e \eta) \widehat{\chi}(r_\e(\xi- \eta)) \deta \right| \\
& = \left |\frac{r^d_\e}{(2\pi)^d}  \int_{\R^d}(1 -  \psi(\e \eta)) \widehat{\chi}(r_\e(\xi- \eta)) \deta \right| \\
&\leq  \frac{r^d_\e}{(2\pi)^d}  \int_{|\eta| \geq \frac{1}{\e}} |\widehat{\chi}(r_\e(\xi- \eta))| \deta \\
&= \frac{1}{(2\pi)^d}  \int_{\left| \xi - \frac{t}{r_\e} \right| \geq \frac{1}{\e}} |\widehat{\chi}(t)| \dt.
\end{align*} 
By \cite[Lemma 3.3]{Komatsu} there is $C > 0$ such that
\[  |\widehat{\chi}(t)| \leq Ce^{-N(2H^2\lambda t)} \leq ACe^{-2N(2H\lambda t)}, \qquad t \in \R^d. \]
Furthermore, notice that for $\xi, t \in \R^d$ it holds that
\[ |\xi| \leq \frac{1}{2\e} \mbox{ and } \left| \xi - \frac{t}{r_\e} \right| \geq \frac{1}{\e} \to |t| \geq \frac{r_\e}{2\e}. \]
Hence we obtain that
\[ |1 - \widehat{\theta}_\e(\xi)| \leq C'e^{-N(H\lambda r_\e/\e)},\qquad |\xi| \leq \frac{1}{2\e}, \]
where 
\[ C' = \frac{AC}{(2\pi)^d}\int_{\R^d} e^{-N(2H\lambda t)} \dt < \infty. \]
By \Cref{M_p-N_p} there is $\e_0 > 0$ such that 
\[ M(t) \leq N(r_\e t) + M(\lambda/ \e), \qquad t > 0, \]
for all $\e < \e_0$. By setting $t = H\lambda/\e$ we obtain that
\[ N(H\lambda r_\e/\e) \geq M(H\lambda/\e) - M(\lambda/ \e) \geq M(\lambda /\e) - \log A, \]
for all $\e < \e_0$ and the result follows.
\end{proof}
\begin{proof}[Proof of \Cref{RO-ultradistributions}.]
For $\e \in I$ fixed we have that $\Phi_\e \in \Linc(\mathcal{D}'^\ast(\Omega), \mathcal{E}^\ast(\Omega))$ by \cite[Prop.\ 6.10]{Komatsu}.
The fact that  $(\Phi_\e)_\e$ is localizing follows easily from $\displaystyle \lim_{\e \to 0^+}r_\e = 0$. We now show that $(\Phi_\e)_\e$ satisfies $(\TO)_j$, $j = 1,2,3$. In the Roumieu case we use \Cref{test-objects-Roumieu}. 
$(\TO)_1$: We need to show that 
\begin{gather*}
\forall u \in  \mathcal{D}'^{(M_p)}(\Omega) \, \forall K \Subset \Omega \, \forall h > 0 \, \exists \lambda > 0 :  \| \Phi_\e(u)\|_{K,h}  = O(e^{M(\lambda/\e)}), \\
(\forall u \in  \mathcal{D}'^{\{M_p\}}(\Omega) \, \forall K \Subset \Omega \, \forall \lambda > 0 \, \exists h > 0 : \| \Phi_\e(u)\|_{K,h}  = O(e^{M(\lambda/\e)})).
\end{gather*}
There is $N \in \N$ such that  
\[ \supp \theta_\e( x - \cdot ) \subseteq K_N, \qquad x \in K, \]
for $\e$ small enough. Hence 
\[ \Phi_\e(u)(x) = (\kappa u \ast \theta_\e)(x),  \qquad x \in K, \]
where $\kappa = \kappa_N$, for $\e$ small enough. By \cite[Lemma 3.3]{Komatsu} it suffices to show that for all $h > 0$ there is  $\lambda > 0$ (for all $\lambda > 0$ there is  $h > 0$) such that
\[ \int_{\R^d} |\widehat{\kappa u}(\xi)| |\widehat{\theta}_\e(\xi)|e^{M(\xi/h)} \dxi = O(e^{M(\lambda/\e)}). \] 
There are $\mu > 0$ and $C > 0$ (for every $\mu > 0$ there is $C > 0$) such that
\[|\widehat{\kappa u}(\xi)| \leq C e^{M(\mu\xi)}, \qquad \xi \in \R^d. \]
\Cref{growth-fourier-theta} \ref{7.5.2} implies that for all $h,\lambda > 0$ (both in the Beurling and Roumieu case)
\[ \int_{|\xi| \geq \frac{4}{\e}} |\widehat{\kappa u}(\xi)| |\widehat{\theta}_\e(\xi)|e^{M(\xi/h)} \dxi = O(e^{M(\lambda/\e)}). \] 
On the other hand, by \Cref{growth-fourier-theta} \ref{7.5.1}, we have that
\begin{align*}
\int_{|\xi| \leq \frac{4}{\e}} |\widehat{\kappa u}(\xi)| |\widehat{\theta}_\e(\xi)|e^{M(\xi/h)} \dxi &\leq \frac{AC\|\widehat{\chi}\|_{L^1(\R^d)}}{(2\pi)^d} \int_{|\xi| 
\leq \frac{4}{\e}}e^{M(\mu\xi)+ M(H\xi/h) -M(\xi/h)} \dxi \\
&\leq C'e^{M(\lambda_0 /\e)}
\end{align*}
where $\lambda_0 = 4H\max(\mu, H/h)$ and 
\[ C' = \frac{AC\|\widehat{\chi}\|_{L^1(\R^d)}}{(2\pi)^d}\int_{\R^d} e^{-M(\xi/h)} \dxi < \infty. \]
The Beurling case follows at once while the Roumieu case follows by noticing that $\lambda_0$ can be made as small as desired by choosing $\mu$ small enough and $h$ big enough.

$(\TO)_2$: By \cite[Prop. 4.2]{D-V-V} it suffices to show that 
\begin{gather*}
\forall \varphi \in  \mathcal{E}^{(M_p)}(\Omega) \, \forall K \Subset \Omega \,  \forall \lambda > 0 :  \sup_{x \in K}| \Phi_\e(\varphi)(x) - \varphi(x)|  = O(e^{-M(\lambda/\e)}), \\
(\forall \varphi \in  \mathcal{E}^{\{M_p\}}(\Omega) \, \forall K \Subset \Omega \,  \exists \lambda > 0 : \sup_{x \in K}| \Phi_\e(\varphi)(x) - \varphi(x)|  = O(e^{-M(\lambda/\e)})).
\end{gather*}
There is $N \in \N$ such that  
\[ \supp \theta_\e( x - \cdot ) \subseteq K_N, \qquad x \in K, \]
for $\e$ small enough. Hence 
\[ \Phi_\e(\varphi)(x) - \varphi(x) = (\kappa \varphi \ast \theta_\e)(x) - \kappa(x)\varphi(x),  \qquad x \in K, \]
where $\kappa = \kappa_N$, and, thus, 
\begin{align*}
\sup_{x \in K}| \Phi_\e(\varphi)(x) - \varphi(x)|  
&= \sup_{x \in K} |(\kappa \varphi \ast \theta_\e)(x) - \kappa(x)\varphi(x)| \\
&\leq \frac{1}{(2\pi)^d} \int_{\R^d} |\widehat{\kappa \varphi}(\xi)| |1 - \widehat{\theta}_\e(\xi)| \dxi
\end{align*}
for $\e$ small enough. Therefore it suffices to show that for all $\lambda > 0$ (for some $ \lambda > 0$) it holds that
\[ \int_{\R^d} |\widehat{\kappa \varphi}(\xi)| |1 - \widehat{\theta}_\e(\xi)| \dxi = O(e^{-M(\lambda/\e)}). \] 
For every $\mu > 0$ there is $C > 0$ (there are $\mu, C > 0$) such that
\[ |\widehat{\kappa \varphi}(\xi)| \leq C e^{-M(H\mu\xi)} \leq ACe^{-2M(\mu \xi)}, \qquad \xi \in \R^d. \]
\Cref{growth-fourier-theta} \ref{7.5.3} implies that for all $\lambda > 0$ (both in the Beurling and Roumieu case)
\[ \int_{|\xi| \leq \frac{1}{2\e}} |\widehat{\kappa \varphi}(\xi)| |1 - \widehat{\theta}_\e(\xi)| \dxi = O(e^{-M(\lambda/\e)}). \] 
On the other hand, by \Cref{growth-fourier-theta} \ref{7.5.1}, we have that
\begin{align*}
\int_{|\xi| \geq \frac{1}{2\e}} |\widehat{\kappa \varphi}(\xi)| |1 - \widehat{\theta}_\e(\xi)| \dxi \leq C'e^{-M(\mu /(2\e))}
\end{align*}
where 
\[ C' = \left( 1 + \frac{\|\widehat{\chi}\|_{L^1(\R^d)}}{(2\pi)^d}\right)AC\int_{\R^d} e^{-M(\mu\xi)} \dxi < \infty. \]

$(\TO)_3$: Since the space $\mathcal{D}^*(\Omega)$ is Montel it suffices to show that for all $u \in \mathcal{D}'^*(\Omega)$ it holds that
\[ \lim_{\e \to 0^+} \int_{\R^d}\Phi_\e(u)(x) \varphi(x) \dx = \langle u, \varphi \rangle, \qquad \varphi \in \mathcal{D}^*(\Omega). \]
There is $N \in \N$ such that  
\[ \supp \theta_\e( x - \cdot ) \subseteq K_N, \qquad x \in \supp \varphi, \]
for $\e$ small enough. Hence, for $\kappa = \kappa_N$, we have that
\begin{align*}
\int_{\R^d}\Phi_\e(u)(x) \varphi(x) \dx &= \int_{\R^d} \langle u(y),  \kappa(y) \theta_\e(x - y) \rangle \varphi(x) \dx \\
&= \langle u(y), \kappa(y)\int_{\R^d}\theta_\e(x - y)\varphi(x) \dx \rangle \\
&=  \langle u(y), \kappa(y) \Phi_\e(\varphi)(y) \rangle \\
\end{align*}
for $\e$ small enough. The result now follows from $(\TO)_2$ and the continuity of $u$.
\end{proof}

As in \Cref{sec_appldist} we now obtain a fine sheaf $\cGsloc$ of algebras such that $\sigma \colon \cE^* \to \cGsloc$ is a sheaf homomorphism of algebras and $\iota \colon \cD^{*\prime}$ is a sheaf homomorphism of vector spaces.

The partial derivative $\partial_i$, $i = 1, \ldots, n$, satisfies the assumptions of \Cref{extension-der}, so it defines a mapping
\[ \widehat \partial_i \colon \cG^\ast_{\operatorname{loc}}(\Omega) \to \cG^\ast_{\operatorname{loc}}(\Omega) \]
which commutes with $\iota$.

Moreover, as seen in \cite[p.~626]{Komatsu2} real analytic coordinate transformations induce continuous mappings on the spaces $\cD^{*\prime}$ and $\cE^*$, so by \Cref{diff-global} we obtain corresponding actions on the quotient spaces $\cGsloc$.

\begin{theorem}\label{thmtwo}For each open set $\Omega \subseteq \R^n$ there is an associative commutative algebra with unit $\cGsloc(\Omega)$ containing $\cD^{*\prime}(\Omega)$ injectively 
as a linear subspace and $\cE^*(\Omega)$ as a subalgebra. $\cG^\ast_{\operatorname{loc}}(\Omega)$ is a differential algebra, where the partial derivatives $\widehat \partial_i$, $i = 1,\ldots,n$, extend the usual partial derivatives from $\cD^\ast(\Omega)$ to $\cG^\ast_{\operatorname{loc}}
(\Omega)$, and $\cGsloc$ is a fine sheaf of algebras over $\Omega$. Moreover, the construction is invariant under real-analytic coordinate changes, i.e., if $\mu \colon \Omega' \to \Omega$ is a real-analytic diffeomorphism then there is a map $\widehat \mu 
\colon \cGsloc(\Omega') \to \cGsloc(\Omega)$ compatible with the canonical embeddings $\iota$ and $\sigma$. 
\end{theorem}


\newcommand{\doi}[1]{\textsc{doi:} \href{http://dx.doi.org/#1}{\texttt{#1}}}
\newcommand{\isbn}[1]{\textsc{isbn:} #1}
\newcommand{\issn}[1]{\textsc{issn:} #1}


\begin{thebibliography}{20}
\bibitem{zbMATH06071218} K. Benmeriem and C. Bouzar. ``Generalized Gevrey ultradistributions and their
microlocal analysis.'' In: \textit{Pseudo-differential operators: analysis, applications and
computations}. Vol. 213. Oper. Theory Adv. Appl. Basel: Birkhäuser, 2011, pp. 235--250. \isbn{978-3-0348-0048-8}. \textsc{doi:} \href{http://dx.doi.org/10.1007/978-3-0348-0049-5_14}{\texttt{10.1007/978-3-0348-0049-5\textunderscore14}}.

\bibitem{ColNew} J. F. Colombeau. \textit{New generalized functions and multiplication of distributions}.
North-Holland Mathematics Studies 84. Amsterdam: North-Holland Publishing
Co., 1984. \isbn{978-0-444-86830-5}.

\bibitem{ColElem} J. F. Colombeau. \textit{Elementary introduction to new generalized functions}. North-Holland Mathematics Studies 113. Amsterdam: North-Holland Publishing Co.,
1985. \isbn{0-444-87756-8}.

\bibitem{D-V-V} A. Debrouwere, H. Vernaeve, and J. Vindas. ``Optimal embeddings of ultradistributions into differential algebras''. Monatsh. Math. (2017). In press. \doi{10.1007/s00605-017-1066-6}.

\bibitem{zbMATH02134784} A. Delcroix, M. F. Hasler, S. Pilipovi\'c, and V. Valmorin. ``Embeddings of ultradistributions and periodic hyperfunctions in Colombeau type algebras through
sequence spaces''. \textit{Math. Proc. Camb. Philos. Soc.} 137.3 (2004), 697--708.
\issn{0305-0041}. \doi{10.1017/S0305004104007923}.

\bibitem{1126.46030} A. Delcroix, M. F. Hasler, S. Pilipovi\'c, and V. Valmorin. ``Sequence spaces with
exponent weights. Realizations of Colombeau type algebras''. \textit{Diss. Math.} 447
(2007). \doi{10.4064/dm447-0-1}.

\bibitem{G-B} P. Giordano and L. Luperi Baglini. ``Asymptotic gauges: generalization of Colombeau
type algebras''. \textit{Math. Nachr.} 289.2-3 (2016), 247--274. \issn{0025-584X}.

\bibitem{zbMATH04205141} T. Gramchev. ``Nonlinear maps in spaces of distributions''. \textit{Math. Z.} 209.1 (1992),
101--114. \issn{0025-5874}. \doi{10.1007/BF02570824}.

\bibitem{found} M. Grosser, E. Farkas, M. Kunzinger, and R. Steinbauer. ``On the foundations
of nonlinear generalized functions I, II''. \textit{Mem. Am. Math. Soc.} 729 (2001). \issn{0065-9266}. \doi{10.1090/memo/0729}.

\bibitem{GKOS} M. Grosser, M. Kunzinger, M. Oberguggenberger, and R. Steinbauer. \textit{Geometric
theory of generalized functions with applications to general relativity}. Mathematics
and its Applications 537. Dordrecht: Kluwer Academic Publishers, 2001. \isbn{1-4020-0145-2}.

\bibitem{global} M. Grosser, M. Kunzinger, R. Steinbauer, and J. A. Vickers. ``A Global Theory
of Algebras of Generalized Functions''. \textit{Adv. Math.} 166.1 (2002), 50--72. \issn{0001-8708}. \doi{10.1006/aima.2001.2018}.

\bibitem{Komatsu} H. Komatsu. ``Ultradistributions. I: Structure theorems and a characterization''.
\textit{J. Fac. Sci., Univ. Tokyo, Sect. I A} 20 (1973), 25--105. \issn{0040-8980}.

\bibitem{Komatsu2} H. Komatsu. ``Ultradistributions. II: The kernel theorem and ultradistributions
with support in a submanifold''. \textit{J. Fac. Sci., Univ. Tokyo, Sect. I A} 24 (1977),
607--628. \issn{0040-8980}.

\bibitem{Komatsu3} H. Komatsu. ``Ultradistributions. III: Vector valued ultradistributions and the
theory of kernels''. \textit{J. Fac. Sci., Univ. Tokyo, Sect. I A} 29 (1982), 653--717.
\issn{0040-8980}.

\bibitem{KM} A. Kriegl and P. W. Michor. \textit{The convenient setting of global analysis}. Mathematical Surveys and Monographs 53. Providence, RI: American Mathematical
Society, 1997. \isbn{0-8218-0780-3}.

\bibitem{gf14proc} E. A. Nigsch. ``Some extensions to the functional analytic approach to Colombeau
algebras''. \textit{Novi Sad J. Math.} 45.1 (2015), 231--240.

\bibitem{papernew} E. A. Nigsch. ``The functional analytic foundation of Colombeau algebras''. \textit{J. Math. Anal. Appl.} 421.1 (2015), 415--435. \doi{10.1016/j.jmaa.2014.07}.
014.

\bibitem{bigone} E. A. Nigsch. ``Nonlinear generalized sections of vector bundles''. \textit{J. Math. Anal.
Appl.} 440 (2016), 183--219. \issn{0022-247X}. \doi{10.1016/j.jmaa.2016.
03.022}.

\bibitem{MOBook} M. Oberguggenberger. \textit{Multiplication of Distributions and Applications to Partial Differential Equations}. Pitman Research Notes in Mathematics 259. Harlow,
U.K.: Longman, 1992. \isbn{978-0-582-08733-0}.

\bibitem{zbMATH01618580} S. Pilipovi\'c and D. Scarpalezos. ``Colombeau generalized ultradistributions''. \textit{Math.
Proc. Camb. Philos. Soc.} 130.3 (2001), 541--553. \issn{0305-0041}. \doi{10.1017/S0305004101005072}.

\bibitem{TD} L. Schwartz. \textit{Th\'eorie des distributions}. Nouvelle \'edition, enti\`erement corrig\'ee,
refondue et augment\'ee. Paris: Hermann, 1966.
\end{thebibliography}

\end{document}